\documentclass{amsart}
\usepackage{amsmath}
\usepackage{amssymb}
\usepackage{amsthm}
\usepackage{enumerate}
\usepackage[pdftex]{graphicx}
\usepackage{caption}
\theoremstyle{definition}
\newtheorem{definition}{Definition}[section]
\theoremstyle{plain}
\newtheorem{lemma}[definition]{Lemma}
\newtheorem{theorem}[definition]{Theorem}

\newtheorem{proposition}[definition]{Proposition}
\newtheorem{corollary}[definition]{Corollary}
\theoremstyle{remark}
\newtheorem{remark}[definition]{Remark}

\newtheorem{example}[definition]{Example}

\makeatletter
\@namedef{subjclassname@2020}{
  \textup{2020} Mathematics Subject Classification}
\makeatother

\newcommand{\myint}{\operatorname{int}}
\newcommand{\mysucc}{\operatorname{succ}}
\newcommand{\myindr}{\mathcal R_{\text{ind}}}

\begin{document}

\title[almost o-minimal structures and $\mathfrak X$-structures]{Almost o-minimal structures and $\mathfrak X$-structures}
\author[M. Fujita]{Masato Fujita}
\address{Department of Liberal Arts,
Japan Coast Guard Academy,
5-1 Wakaba-cho, Kure, Hiroshima 737-8512, Japan}
\email{fujita.masato.p34@kyoto-u.jp}

\begin{abstract}
We propose new structures called almost o-minimal structures and $\mathfrak X$-structures.
The former is a first-order expansion of a dense linear order without endpoints such that the intersection of a definable set with a bounded open interval is a finite union of points and open intervals.
The latter is a variant of van den Dries and Miller's analytic geometric categories and Shiota's $\mathfrak X$-sets and $\mathfrak Y$-sets.
In them, the family of definable sets are closed only under proper projections unlike first-order structures.
We demonstrate that an $\mathfrak X$-expansion of an ordered divisible abelian group always contains an o-minimal expansion of an ordered group such that all bounded $\mathfrak X$-definable sets are definable in the structure.

Another contribution of this paper is a uniform local definable cell decomposition theorem for almost o-minimal expansions of ordered groups $\mathcal M=(M,<,0,+,\ldots)$.
Let $\{A_\lambda\}_{\lambda\in\Lambda}$ be a finite family of definable subsets of $M^{m+n}$.
Take an arbitrary positive element $R \in M$ and set $B=]-R,R[^n$.
Then, there exists a finite partition into definable sets 
\begin{equation*}
M^m \times B = X_1 \cup \ldots \cup X_k
\end{equation*}
such that $B=(X_1)_b \cup \ldots \cup (X_k)_b$ is a definable cell decomposition of $B$ for any $b \in M^m$ and either $X_i \cap A_\lambda = \emptyset$ or $X_i \subseteq A_\lambda$ for any $1 \leq i \leq k$ and $\lambda \in \Lambda$.
Here, the notation $S_b$ denotes the fiber of a definable subset $S$ of $M^{m+n}$ at $b \in M^m$.
We introduce the notion of multi-cells and demonstrate that any definable set is a finite union of multi-cells in the course of the proof of the above theorem.
\end{abstract}

\subjclass[2020]{Primary 03C64; Secondary 14P99}

\keywords{almost o-minimal structure, $\mathfrak X$-structure, semi-definable set}

\maketitle

\section{Introduction}\label{sec:intro}
O-minimal structures \cite{vdD, KPS,PS} have been studied model theoretically and geometrically.
An expansion of a dense linear order without endpoints $\mathcal M=(M,<,\ldots)$ is \textit{o-minimal} if any definable subset of $M$ is a finite union of points and open intervals.
Studies on o-minimal structures are too many to be presented here. 
One of main interests in studying o-minimal structures is their tame topology.
They possess various tame topological properties such as monotonicity theorem and definable cell decomposition theorem. 

An interesting question is what topological properties are remained when the definition of o-minimal structures is relaxed. 
In fact, many structures relaxing the definition of o-minimal structures have been proposed and their topological properties have been investigated.
Here is an incomplete list; weakly o-minimal structures \cite{MMS, W}, structures having o-minimal open core \cite{DMS, F}, d-minimal structures \cite{M2,T}, locally o-minimal structures \cite{TV, KTTT}, models of DCTC \cite{S} and uniformly locally o-minimal structures of the second kind \cite{Fuji}.
We propose a new relative of these structures named an \textit{almost o-minimal structure} in this paper.
Why do we propose a new structure though many structures have been already proposed?
We explain why.
The notation $\mathcal M$ denotes a structure and $M$ denotes its universe below.

Toffalori and Vozoris proposed a locally o-minimal structure \cite{TV}, which is defined by simply localizing the definition of an o-minimal structure.
An expansion of a dense linear order without endpoints $\mathcal M=(M,<,\ldots)$ is \textit{locally o-minimal} if, for any definable subset $X$ of $M$ and any point $x \in M$, there exists an open interval $I$ containing the point $x$ such that the intersection $I \cap X$ is a finite union of points and open intervals.
In spite of its similarity to the definition of o-minimal structures, a locally o-minimal structure does not enjoy the localized properties possessed by o-minimal structures.
Schoutens introduced a \textit{model of DCTC} generalizing a locally o-minimal expansion of an ordered field \cite{S}.
Roughly speaking, a model of DCTC is a locally o-minimal structure which is o-minimal at the infinities $\pm \infty$.
More precisely, for any set $X$ in $M$ definable in a model of DCTC, there exist $a,b \in M$ such that $X \cap \{x<a\}$ and $X \cap \{x>b\}$ are empty sets or open intervals. 
A locally o-minimal expansion of an ordered field and a model of DCTC possess several tame topological properties.
Readers who are interested in them should consult \cite{F,S, Fuji4}. 

The author have pursued another direction. 
His initial purpose is to find a necessary and sufficient condition for a locally o-minimal structure admitting a local definable cell decomposition \cite{Fuji}.
The answer was a uniformly locally o-minimal structure of the second kind when the structure is definably complete.
\begin{definition}\label{def:second}
We consider an expansion $\mathcal M=(M,<,\ldots)$ of a dense linear order without endpoints.
It is \textit{definably complete} if every definable subset of $M$ has both a supremum and an infimum in $M \cup \{ \pm \infty\}$ \cite{M}.
A definably complete expansion of an ordered group is divisible and abelian \cite[Proposition 2.2]{M}.

A locally o-minimal structure $\mathcal M=(M,<,\ldots)$ is a \textit{uniformly locally o-minimal structure of the second kind} if, for any positive integer $n$, any definable set $X \subseteq M^{n+1}$, $a \in M$ and $b \in M^n$, there exist an open interval $I$ containing the point $a$ and an open box $B$ containing $b$ such that the definable sets $X_y \cap I$ are finite unions of points and open intervals for all $y \in B$.
Here, $X_y$ denotes the fiber $\{x \in M\;|\; (y,x) \in X\}$.
When we can choose $B=M^n$, the structure $\mathcal M$ is called a \textit{uniformly locally o-minimal structure of the first kind}.
\end{definition}
We frequently consider definably complete uniformly locally o-minimal expansions of the second kind of ordered groups.
We simply call them \textit{DCULOAS structures}.
Their local tame topological properties have been clarified in a series of papers \cite{Fuji, Fuji3, Fuji4}.

In many potential applications to other mathematical branches such as geometry and analysis, the universe is the set of reals $\mathbb R$.
The author also demonstrated that a locally o-minimal expansion of the ordered group of reals admits local definable cell decomposition better than a general definably complete uniformly locally o-minimal structure of the second kind in an unpublished paper \cite{Fuji2}, which is a special case of Theorem \ref{thm:udcd}.
A DCULOAS structure is not an excellent abstraction of locally o-minimal expansion of the ordered group of reals.
The significant difference between the real case and the general case is that any bounded definable set is a finite union of points and open intervals in the former, but it may not be in the latter. 
This is the reason why we focus the following notion:
\begin{definition}
An expansion $\mathcal M=(M,<,\ldots)$ of densely linearly ordered set without endpoints is \textit{almost o-minimal} if any bounded definable set in $M$ is a finite union of points and open intervals.
\end{definition}
Note that an locally o-minimal expansion of the ordered set of reals $(\mathbb R,<)$ is almost o-minimal.
Roughly speaking, an almost o-minimal structure is o-minimal on bounded regions.
The notion of almost o-minimality is a complementary notion of DCTC in a sense.
A locally o-minimal structure is o-minimal if and only if it is simultaneously a model of DCTC and almost o-minimal as demonstrated in Proposition \ref{prop:almost1}.

The notion of subanalytic sets is another useful geometrical concept \cite{BM, H}.
A subset of $X$ of $\mathbb R^n$ is \textit{subanalytic} if each point of $\mathbb R^n$ has a neighborhood $U$ such that $X \cap U$ is a finite union of sets of the form $\operatorname{Im}(f_1) \setminus \operatorname{Im}(f_2)$, where $f_1$ and $f_2$ are proper real analytic maps from real analytic manifolds to $\mathbb R^n$.
The projection image of a subanalytic set is not necessarily subanalytic, but
its image under a proper projection is again subanalytic.
The family of subanalytic sets are not the family of sets definable in a first-order language because it is not closed under taking the projection image.

In \cite{vdDM}, van den Dries and Miller generalized the notion of subanalytic sets and proposed an \textit{analytic-geometric category} and clarified the relation between the analytic-geometric category of subanalytic sets and  an o-minimal structure which is called the restricted analytic field $\mathbb R_{\text{an}}$.
Shiota also proposed \textit{$\mathfrak X$-sets} and \textit{$\mathfrak Y$-sets} in \cite{Shiota}.
They are also generalization of subanalytic sets.
The family of sets `definable' in them is only closed under taking the image under a proper projection.
In addition, their underlying set is the set of reals $\mathbb R$.
We want to generalize their concepts when underlying set is a densely linearly ordered set without endpoints.
We propose the following structure generalizing Shiota's $\mathfrak X$-sets and $\mathfrak Y$-sets.

\begin{definition}\label{def:x}
Let $(M,<)$ be a densely linearly ordered set without endpoints.
A map $p$ from a subset $X$ of $M^m$ to $M^n$ is \textit{proper} if the inverse image $ p^{-1}(U)$ of an arbitrary bounded open box $U$ in $M^n$ is bounded.

An \textit{$\mathfrak X$-structure} is a triple $\mathcal X = (M,<,\mathcal S=\{\mathcal S_n\}_{n \in \mathbb N})$ of a densely linearly ordered set without endpoints $(M,<)$ and the families $\mathcal S_n$ of subsets in $M^n$ satisfying the following conditions:
\begin{enumerate}
\item[(1)] For all $x \in M$, the singletons $\{x\}$ belong to $S_1$. 
All open intervals also belong to $S_1$.
\item[(2)] The sets $\{(x,y) \in M^2\;|\; x= y\}$ and $\{(x,y) \in M^2\;|\; x< y\}$ belong to $S_2$.
\item[(3)] $\mathcal S_n$ is a boolean algebra and $M^n \in \mathcal S_n$;
\item[(4)] We have $X_1 \times X_2 \in \mathcal S_{m+n}$ whenever $X_1 \in \mathcal S_m$ and $X_2 \in \mathcal S_n$;
\item[(5)] For any permutation $\sigma$ of $\{1, \ldots, n\}$, the image $\widetilde{\sigma}(X)$ belongs to $\mathcal S_n$ when $X \in \mathcal S_n$ and the notation $\widetilde{\sigma}:M^n \rightarrow M^n$ denotes the map given by $\widetilde{\sigma}(x_1,\ldots, x_n)=(x_{\sigma(1)},\ldots,x_{\sigma(n)})$; 
\item[(6)] Let $\pi:M^n \rightarrow M^m$ be a coordinate projection and $X \in \mathcal S_n$ such that the restriction $\pi|_{X}$ of $\pi$ to $X$ is proper.
Then, the image $\pi(X)$ belongs to $\mathcal S_m$.
\item[(7)] The intersection $I \cap X$ is a finite union of points and open intervals when $X \in \mathcal S_1$ and $I$ is a bounded open interval.
\end{enumerate}
The set $M$ is called the \textit{universe} and the \textit{underlying set} of the $\mathfrak X$-structure $\mathcal X$.
A subset $X$ of $M^n$ is called \textit{$\mathfrak X$-definable} in $\mathcal X$ when $X$ is an element of $\mathcal S_n$.
A set $\mathfrak X$-definable in $\mathcal X$ is simply called $\mathfrak X$-definable when  $\mathcal X$ is clear from the context.
A map from a subset of $M^m$ to $M^n$ is \textit{$\mathfrak X$-definable} if its graph is $\mathfrak X$-definable.

When $(M,<,0,+)$ is an ordered divisible abelian group and the addition is $\mathfrak X$-definable, we call the $\mathfrak X$-structure an \textit{$\mathfrak X$-expansion of an ordered divisible abelian group}.
We define an \textit{$\mathfrak X$-expansion of an ordered real closed field} in the same manner.
\end{definition}
In Shiota's formulation, an $\mathfrak X$-set is locally a finite union of points and open intervals.
Here, we call that $X$ is locally a finite union of points and open intervals when, for any point $x \in \mathbb R$, there exists an open interval $I$ containing the point $x$ such that $I \cap X$ is a finite union of points and open intervals.
The formulation by van den Dries and Miller is similar.
If a subset $X$ of $\mathbb R$ is locally a finite union of points and open intervals, it satisfies the condition (7) in Definition \ref{def:x} because the closed bounded interval is compact in $\mathbb R$.
But, it is not true in a general densely linearly ordered set without endpoints $(M,<)$.
In Shiota's original formulation, we cannot deduce several good properties enjoyed by the $\mathfrak X$-structure defined in our formulation when the underlying set is a general $M$.

An almost o-minimal structure is an $\mathfrak X$-structure.
The following is another important example of $\mathfrak X$-structures.
\begin{definition}
Let $\mathcal R=(M,<,\ldots)$ be an o-minimal structure.
A subset $X$ of $M^n$ is \textit{semi-definable in $\mathcal R$} if the intersection $U \cap X$ is definable in $\mathcal R$ for any bounded open box $U$ in $M^n$. 
A map from a subset of $M^m$ to $M^n$ is \textit{semi-definable} if its graph is semi-definable.
The family $\mathcal S(\mathcal R)=\{\mathcal S(\mathcal R)_n\}_{n \in \mathbb N}$ of all semi-definable sets satisfies the conditions in Definition \ref{def:x}.
The $\mathfrak X$ structure $\mathfrak X(\mathcal R)=(M,<,\mathcal S(\mathcal R))$ is called the \textit{$\mathfrak X$-structure of semi-definable sets in $\mathcal R$}.
\end{definition} 

We study general $\mathfrak X$-structures in Section \ref{sec:x}.
The main theorems of this section are the structure theorems Theorem \ref{thm:in_omin} and Theorem \ref{thm:xstr}.
The former says that an $\mathfrak X$-expansion of an ordered divisible abelian group always contains an o-minimal expansion $\mathcal R$ of an ordered group such that all bounded $\mathfrak X$-definable sets are definable in the structure $\mathcal R$.
The latter gives a sufficient condition for an $\mathfrak X$-expansion of an ordered divisible abelian group being an $\mathfrak X$-expansion of an ordered real closed field.
The basic property of dimension of $\mathfrak X$-definable sets are also investigated in this section.

The $\mathfrak X$-structures of semi-definable sets in an o-minimal structure are studied in Section \ref{sec:semi-definable}.
The notion of semi-definable connectedness is introduced in this section.
The main theorem of this section is Theorem \ref{thm:connected}, which gives equivalent conditions for a semi-definable set to be semi-definably connected and also demonstrates the existence of semi-definably connected components.

Section \ref{sec:almost} is devoted for the study of almost o-minimal structures.
After we investigate the basic properties of almost o-minimal structures, we prove a uniform local definable cell decomposition theorem.
It is the last main theorem of this paper.
The definition of cells and the local definable cell decomposition theorem for a definably complete uniformly locally o-minimal structures of the second kind are as follows:

\begin{definition}[Definable cell decomposition]
Consider an expansion of dense linear order without endpoints $\mathcal M=(M,<,\ldots)$.
Let $(i_1, \ldots, i_n)$ be a sequence of zeros and ones of length $n$.
\textit{$(i_1, \ldots, i_n)$-cells} are definable subsets of $M^n$ defined inductively as follows:
\begin{itemize}
\item A $(0)$-cell is a point in $M$ and a $(1)$-cell is an open interval in $M$.
\item An $(i_1,\ldots,i_n,0)$-cell is the graph of a definable continuous function defined on an $(i_1,\ldots,i_n)$-cell.
An $(i_1,\ldots,i_n,1)$-cell is a definable set of the form $\{(x,y) \in C \times M\;|\; f(x)<y<g(x)\}$, where $C$ is an $(i_1,\ldots,i_n)$-cell and $f$ and $g$ are definable continuous functions defined on $C$ with $f<g$.
\end{itemize}
A \textit{cell} is an $(i_1, \ldots, i_n)$-cell for some sequence $(i_1, \ldots, i_n)$ of zeros and ones.
The sequence $(i_1, \ldots, i_n)$ is called the \textit{type} of an $(i_1, \ldots, i_n)$-cell.
An \textit{open cell} is a $(1,1, \ldots, 1)$-cell.
The dimension of an $(i_1, \ldots, i_n)$-cell is defined by $\sum_{j=1}^n i_j$.

We inductively define a \textit{definable cell decomposition} of an open box $B \subseteq M^n$.
For $n=1$, a definable cell decomposition of $B$ is a partition $B=\bigcup_{i=1}^m C_i$ into finite cells.
For $n>1$, a definable cell decomposition of $B$ is a partition $B=\bigcup_{i=1}^m C_i$ into finite cells such that $\pi(B)=\bigcup_{i=1}^m \pi(C_i)$ is a definable cell decomposition of $\pi(B)$, where $\pi:M^n \rightarrow M^{n-1}$ is the projection forgetting the last coordinate.
Consider a finite family $\{A_\lambda\}_{\lambda \in \Lambda}$ of definable subsets of $B$.
A \textit{definable cell decomposition of $B$ partitioning $\{A_\lambda\}_{\lambda \in \Lambda}$} is a definable cell decomposition of $B$ such that the definable sets $A_{\lambda}$ are unions of cells for all $\lambda \in \Lambda$.

\end{definition}

\begin{theorem}[Local definable cell decomposition theorem, {\cite[Theorem 4.2]{Fuji}}]\label{thm:dcd}
Consider a definably complete uniformly locally o-minimal structure of the second kind $\mathcal M=(M,<,\ldots)$.
Let $n$ be an arbitrary positive integer.
Let $\{A_\lambda\}_{\lambda\in\Lambda}$ be a finite family of definable subsets of $M^n$.
For any point $a \in M^n$, there exist an open box $B$ containing the point $a$ and a definable cell decomposition of $B$ partitioning the finite family $\{B \cap A_\lambda\;|\; \lambda \in \Lambda \text{ and }  B \cap A_\lambda \not= \emptyset\}$.
\end{theorem} 

The above theorem says nothing about the relationship between decompositions at two distinct points.
When the considered structure is an almost o-minimal expansion of an ordered group, we can obtain the following uniform local definable cell decomposition theorem:
\begin{theorem}[Uniform local definable cell decomposition]\label{thm:main}
Consider an almost o-minimal expansion of an ordered group $\mathcal M=(M,<,0,+,\ldots)$.
Let $\{A_\lambda\}_{\lambda\in\Lambda}$ be a finite family of definable subsets of $M^{m+n}$.
Take an arbitrary positive element $R \in M$ and set $B=]-R,R[^n$.
Then, there exists a finite partition into definable sets 
\begin{equation*}
M^m \times B = X_1 \cup \ldots \cup X_k
\end{equation*}
such that $B=(X_1)_b \cup \ldots \cup (X_k)_b$ is a definable cell decomposition of $B$ for any $b \in M^m$ and either $X_i \cap A_\lambda = \emptyset$ or $X_i \subseteq A_\lambda$ for any $1 \leq i \leq k$ and $\lambda \in \Lambda$.
Furthermore, the type of the cell $(X_i)_b$ is independent of the choice of $b$ with $(X_i)_b \not= \emptyset$.
Here, the notation $S_b$ denotes the fiber of a definable subset $S$ of $M^{m+n}$ at $b \in M^m$.
\end{theorem}

We introduce the terms and notations used in this paper.
When a first-order structure is fixed, the term `definable' means `definable in the structure with parameters.'
The notation $f|_A$ denotes the restriction of a map $f:X \rightarrow Y$ to a subset $A$ of $X$.
Consider a linearly ordered set without endpoints $(M,<)$.
An open interval is a nonempty set of the form $\{x \in M\;|\; a < x < b\}$ for some $a,b \in M \cup \{\pm \infty\}$.
It is denoted by $]a,b[$ in this paper.
The closed interval is defined similarly and denoted by $[a,b]$.
We use the notations $]a,b]$ and $[a,b[$ for half open intervals.
The set $M$ equips the order topology induced from the order $<$. 
The affine space $M^n$ equips the product topology of the order topology.
We consider these topologies unless otherwise stated.
An open box is the Cartesian product of open intervals.
For a topological space $T$ and its subset $A$, the notations $\overline{A}$, $\myint(A)$, $\partial A$ and $\operatorname{bd}(A)$ denote the closure, interior, frontier and boundary of $A$, respectively.
The notation $|S|$ denotes the cardinality of a set $S$.
It also denotes the absolute value of an element.
This abuse of notation will not confuse readers.

\section{Geometry of $\mathfrak X$-structures}\label{sec:x}
We study $\mathfrak X$-structures in this section.
\subsection{$\mathfrak X$-definable maps}
We first investigate $\mathfrak X$-definable maps.
Note that the domain of definition of an $\mathfrak X$-definable map is not necessarily $\mathfrak X$-definable.
We can easily get the following lemma.
\begin{lemma}\label{lem:restriction}
Consider an $\mathfrak X$-structure whose underlying set is $M$ and an $\mathfrak X$-definable map $\varphi:X \rightarrow M^n$.
Take an $\mathfrak X$-definable subset $Y$ of $X$.
The restriction $\varphi|_Y$ of $\varphi$ to $Y$ is $\mathfrak X$-definable. 
\end{lemma}
\begin{proof}
Easy. We omit the proof.
\end{proof}

We investigate when the image and the inverse image of an $\mathfrak X$-definable set under an $\mathfrak X$-definable map are again $\mathfrak X$-definable.
\begin{lemma}\label{lem:image}
Consider an $\mathfrak X$-structure whose underlying set is $M$.
Let $X$ be an $\mathfrak X$-definable subset of $M^m$ and $\varphi:X \rightarrow M^n$ be an $\mathfrak X$-definable map.
The image $\varphi(X)$ is $\mathfrak X$-definable when $X$ is bounded or $\varphi$ is proper.
\end{lemma}
\begin{proof}
Consider the graph $\Gamma(\varphi)=\{(x,y) \in X \times M^n\;|\; y=\varphi(x)\}$.
Consider the projection forgetting the first $m$ coordinates.
The image $\varphi(X)$ is the projection image of the graph.
We can easily demonstrate that the restriction of the projection to the graph is proper.
\end{proof}

\begin{definition}
Let $(M,<)$ be a linearly ordered set without endpoints.
Let $X$ be a subset of $M^m$ and $f:X \rightarrow M^n$ be a map.
The map $f$ satisfies the \textit{bounded image condition} if the image $f(X \cap V)$ is bounded for any bounded open box $V$ of $M^m$.
\end{definition}

\begin{lemma}\label{lem:inverse}
Consider an $\mathfrak X$-structure whose underlying set is $M$.
Let $X$ and $Y$ be $\mathfrak X$-definable subsets of $M^m$ and $M^n$, respectively.
Let $\varphi:X \rightarrow M^n$ be an $\mathfrak X$-definable map.
The inverse image $\varphi^{-1}(Y)$ is $\mathfrak X$-definable when $Y$ is bounded or $\varphi$ satisfies the bounded image condition.
\end{lemma}
\begin{proof}
The proof is similar to that of Lemma \ref{lem:image}.
We omit the proof.
\end{proof}

\begin{corollary}\label{cor:semialg}
Consider an $\mathfrak X$-structure whose underlying set is $M$.
Let $\varphi:X \rightarrow M$ be an $\mathfrak X$-definable function satisfying the bounded image condition.
Take $c \in M$.
The sets $\{x \in X\;|\;\varphi(x)=c\}$,  $\{x \in X\;|\;\varphi(x)<c\}$ and $\{x \in X\;|\;\varphi(x)>c\}$ are $\mathfrak X$-definable.
\end{corollary}
\begin{proof}
The sets given in the corollary are the inverse images of $\{c\}$, $]-\infty,c[$ and $]c,\infty[$ under the function $\varphi$.
The corollary follows from Lemma \ref{lem:inverse}.
\end{proof}

\begin{example}\label{ex:x2}
Consider a definably complete structure $\mathcal M=(M,<,\ldots)$.
The image of a definable closed and bounded set under a definable continuous map $f$ is again definable, closed and bounded by \cite[Proposition 1.10]{M}.
It means that the map $f$ satisfies the bounded image condition in this case.
It is not true in an $\mathfrak X$-structure.

The ordered field $(\mathbb R_{\text{alg}},<,+,\cdot,0,1)$ of the real numbers algebraic over $\mathbb Q$ is an ordered real closed field and the induced structure is o-minimal.
In particular, the structure is definably complete.
Consider the $\mathfrak X$-structure of semi-definable sets in this o-minimal structure.
The set of positive integers $\mathbb N$ is semi-definable.
Take $a_n,b_n \in \mathbb Q$ so that $a_n < a_{n+1} < \pi < b_{n+1} < b_n$ for all $n \in \mathbb N$ and $\lim_{n \to \infty}a_n = \lim_{n \to \infty}b_n = \pi$.
Here, $\pi$ denotes the pi $=3.14\ldots$.
We define a semi-definable function $f$ on $[a_1,b_1]$.
The graph of the restriction of $f$ to $[a_i,a_{i+1}]$ is the segment connecting the points $(a_i,i)$ and $(a_{i+1},i+1)$ for any $i \in \mathbb N$.
We define the restriction of $f$ to $[b_i,b_{i+1}]$ in the same manner.
The function $f$ is $\mathfrak X$-definable and continuous, but its image is not bounded.
\end{example}

The composition of two $\mathfrak X$-definable map is not necessarily $\mathfrak X$-definable.
We find a sufficient condition for the composition being $\mathfrak X$-definable.

\begin{lemma}\label{lem:composition}
Consider an $\mathfrak X$-structure.
Let $\varphi:X \rightarrow Y$ and $\psi: Y \rightarrow Z$ be two $\mathfrak X$-definable maps.
The composition $\psi \circ \varphi$ is $\mathfrak X$-definable if $\psi$ is proper or $\varphi$ satisfies the bounded image condition.
\end{lemma}
\begin{proof}
Let $M$ be the underlying set of the $\mathfrak X$-structure.
Let $M^l$, $M^m$ and $M^n$ be the ambient spaces of $X$, $Y$ and $Z$, respectively.
Consider the set $A=\{(x,y,z) \in X \times Z \times Y\;|\;z=\varphi(x),\ y=\psi(z)\}$.
It is $\mathfrak X$-definable by Definition \ref{def:x}(3), (4) and (5).
Let $\pi:M^{l+m+n} \rightarrow M^{l+n}$ be the projection forgetting the last $m$ coordinates.
The graph of the composition $\psi \circ \varphi$ is the image of $A$ under the projection $\pi$.
If the restriction of $\pi$ to $A$ is proper, the graph is $\mathfrak X$-definable by Definition \ref{def:x}(6).
We have only to demonstrate that the restriction is proper.

Take a bounded open box $U$ in $M^l$ and a bounded open box $W$ in $M^n$.
We show that $C=\pi^{-1}(U \times W) \cap A$ is bounded.
When $\psi$ is proper, the inverse image $\psi^{-1}(W)$ is bounded.
The set $C$ is contained in $U \times W \times \psi^{-1}(W)$, and it is bounded.
When $\varphi$ satisfies the bounded image condition, the set $\varphi(X \cap U)$ is bounded.
The set $C$ is contained in $U \times W \times \varphi(X \cap U)$, and it is bounded.
\end{proof}

The following lemma is easy to prove.
The proofs are left to readers.
\begin{lemma}\label{lem:bounded}
The following assertions hold true.
\begin{enumerate}
\item[(1)] Consider an ordered group.
The addition satisfies the bounded image condition.
The addition of a constant is proper and satisfies the bounded image condition.
\item[(2)] Consider a divisible abelian group.
Multiplication by a rational constant is proper and satisfies the bounded image condition.
\item[(3)] Consider an ordered field. 
The multiplication satisfies the bounded image condition.
The multiplication by a constant is proper and satisfies the bounded image condition.
\item[(4)] Consider an $\mathfrak X$-structure and let $\varphi:X \rightarrow Y$ and $\psi: Y \rightarrow Z$ be two $\mathfrak X$-definable maps.
The composition $\psi \circ \varphi$ is proper if both $\varphi$ and $\psi$ are proper.
\item[(5)] Let $\varphi$ and $\psi$ be as in (4).
The composition $\psi \circ \varphi$ satisfies the bounded image condition if both $\varphi$ and $\psi$ satisfy the bounded image condition.
\end{enumerate}
\end{lemma}

\begin{corollary}\label{cor:linear}
Consider an $\mathfrak X$-expansion of an ordered divisible abelian group whose underlying set is $M$.
Consider a linear function $l(\overline{x})=\sum_{i=1}^n q_ix_i +c$ with $q_i \in \mathbb Q$ for all $1 \leq i \leq n$ and $c \in M$, where $\overline{x}=(x_1,\ldots, x_n)$.
The function $l(\overline{x})$ is $\mathfrak X$-definable and the sets of the form $\{\overline{x} \in M^n\;|\; l(\overline{x}) * 0\}$ for $* \in \{=,<,>\}$ are $\mathfrak X$-definable.
\end{corollary}
\begin{proof}
The function $l(\overline{x})$ is $\mathfrak X$-definable and satisfies the bounded image condition by Lemma \ref{lem:composition} and Lemma \ref{lem:bounded}.
The sets given in the corollary are $\mathfrak X$-definable by Corollary \ref{cor:semialg}.
\end{proof}

\subsection{O-minimal structure contained in $\mathfrak X$-structure}
Any $\mathfrak X$-structure has an o-minimal structure $\mathcal R$ such that any bounded $\mathfrak X$-definable set is definable in the o-minimal structure $\mathcal R$.
\begin{lemma}\label{lem:in_omin}
Consider an $\mathfrak X$-structure whose underlying set is $M$.
There exists an o-minimal structure $\mathcal R$ having the same underlying set and satisfying the following conditions:
\begin{enumerate}
\item[(i)] Any set definable in $\mathcal R$ is $\mathfrak X$-definable.
\item[(ii)] Any bounded $\mathfrak X$-definable set is definable in $\mathcal R$.
\end{enumerate}
\end{lemma}
\begin{proof}
For any bounded $\mathfrak X$-definable set $X$, we define the predicate $P_X$ and interpret it naturally.
The notation $\mathcal S_{\text{bdd}}$ denote the set of all bounded $\mathfrak X$-sets. 
Consider the language $L=(<,(P_X)_{X \in \mathcal S_{\text{bdd}}})$.
The set $M$ is naturally the underlying set of an $L$-structure $\mathcal R$.
The structure $\mathcal R$ obviously satisfies the condition (ii).
We demonstrate that $\mathcal R$ is an o-minimal structure satisfying the condition (i).

We need a preparation.
An $L$-formula $\phi(\overline{x})$ with parameters in $M$ is called \textit{simple} when it is one of the following formula:
\begin{align*}
& x_i =c,\  x_i <c,\  x_i>c, \ x_i=x_j,\ x_i<x_j \text{ and }x_i > x_j\text{,}
\end{align*}
where $\overline{x}=(x_1,\ldots, x_n)$, $c \in M$ and $1 \leq i <j \leq n$.
An $L$-formula $\phi(\overline{x})$ with parameters in $M$ is \textit{semi-simple} if it is a finite conjunction of simple formulas.
A subset $X$ in $M^n$ definable in $\mathcal R$ is \text{semi-simple} if it is defined by a semi-simple formula. 
An open box is semi-simple and the complement of an open box is a finite union of semi-simple sets.
We demonstrate the following claim:
\medskip

\textbf{Claim.} Any $L$-formula $\phi(\overline{x})$ with parameters in $M$ is equivalent to either a finite disjunction of semi-simple formulas, a formula of the form $P_X(\overline{x})$ or their disjunction in $\operatorname{Th}(\mathcal R)$. 
Here, the notation $\operatorname{Th}(\mathcal R)$ denotes the set of all the $L$-sentences valid in the structure $\mathcal R$.
\medskip

We demonstrate the claim by induction on the complexity of the formula $\phi(\overline{x})$.
The claim is obvious when $\phi(\overline{x})$ is an atomic formula.
The conjunction of $P_X(\overline{x})$ and $P_Y(\overline{x})$ is equivalent to $P_{X \cap Y}(\overline{x})$.
Their disjunction is equivalent to $P_{X \cup Y}(\overline{x})$.
Let $\phi_1(\overline{x})$ and $\phi_2(\overline{x})$ be finite disjunctions of semi-simple formulas.
The conjunction $\phi_1(\overline{x}) \wedge \phi_2(\overline{x})$ is obviously a finite disjunction of semi-simple formulas.
When $\phi_1(\overline{x})$ is a finite disjunction of semi-simple formulas and $X$ is a bounded $\mathfrak X$-definable set, the set $Y=X \cap \{\overline{x} \in M^n\;|\; \mathcal M \models \phi_1(\overline{x})\}$ is a bounded $\mathfrak X$-definable set.
We have $\mathcal R \models \forall \overline{x}\ ((\phi_1(\overline{x}) \wedge P_X(\overline{x})) \leftrightarrow P_Y(\overline{x}))$.
Therefore, the claim is true for the conjunction of two formulas $\phi_1(\overline{x})$ and $\phi_2(\overline{x})$ satisfying the claim.

We consider the case in which $\phi(\overline{x})$ is the negation of the formula $\psi(\overline{x})$ satisfying the claim.
The formula $\phi(\overline{x})$ clearly satisfies the claim when the formula $\psi(\overline{x})$ is equivalent to a finite disjunction of semi-simple formulas.
We next consider the case in which $\psi(\overline{x})=P_X(\overline{x})$ for some bounded $\mathfrak X$-definable subset $X$ of $M^n$.
There exists $a,b$ in $M$ with $a<b$ and $X \subseteq ]a,b[^n$.
The set $]a,b[^n \setminus X$ is a bounded $\mathfrak X$-definable set and it belongs to $\mathcal S_{\text{bdd}}$.
It is obvious that there exists a finite disjunction $\psi'(\overline{x})$ of semi-simple formulas such that $\mathcal R \models \forall \overline{x}\ (\psi'(\overline{x}) \leftrightarrow x \not\in ]a,b[^n)$.
Therefore, we have $\mathcal R \models \forall \overline{x}\ (\neg P_X(\overline{x}) \leftrightarrow (P_{]a,b[^n\setminus X}(\overline{x}) \vee \psi'(\overline{x})))$.
Using these facts, we can demonstrate that $\phi(\overline{x})=\neg\psi(\overline{x})$ satisfies the claim.
We omit the details.
 
 The projection image of a set in $\mathcal S_{\text{bdd}}$ is again an element of $\mathcal S_{\text{bdd}}$
 Using this fact, we can prove the claim when the formula $\phi(\overline{x})$ is of the form $\exists y\ \psi(\overline{x},y)$.
 We also omit the details.
 We have demonstrated the claim.
\medskip

Take an arbitrary subset $X$ of $M^n$ definable in $\mathcal R$.
It is either a finite union of semi-simple sets, a bounded $\mathfrak X$-definable set or their union by the claim.
A semi-simple set is obviously $\mathfrak X$-definable.
Hence, the set $X$ is $\mathfrak X$-definable.
We have demonstrated that the condition (i) holds true.

We finally show that a subset $X$ of $M$ definable in $\mathcal R$ is a finite union of points and open intervals.
A bounded $\mathfrak X$-definable subset of $M$ is a finite union of points and open intervals by Definition \ref{def:x}(7).
A semi-simple subset of $M$ is obviously a finite union of points and open intervals.
Therefore, $X$ is a finite union of points and open intervals by the claim.
\end{proof}

A locally o-minimal structure \textit{admits  local definable cell decomposition} if local definable cell decomposition in Theorem \ref{thm:dcd} is always available.
Note that a locally o-minimal structure which admits local definable cell decomposition is always a uniformly locally o-minimal structure of the second kind.

\begin{corollary}\label{cor:second}
An almost o-minimal structure admits local definable cell decomposition.
In particular, it is a uniformly locally o-minimal structure of the second kind.
\end{corollary}
\begin{proof}
Recall that an almost o-minimal structure is an $\mathfrak X$-structure.
The corollary follows from Lemma \ref{lem:in_omin} and \cite[Chapter 3, Theorem 2.11]{vdD}.
\end{proof}

We quote \cite[Fact 1.7]{E} which is originally proved in \cite[Proposition 5.1(1)]{LP}:
\begin{proposition}\label{prop:lp1}
Let $\mathcal V=(V,+,<,a,(d)_{d \in D},(P)_{P \in \mathcal P})$ be an expansion of an ordered vector space $(V,+,<,(d)_{d \in D})$ over an ordered division ring $D$ by predicates $P \in \mathcal P$ on a bounded subset of $[-a,a]^n$, such that $\mathcal P$ contains predicates for all subsets of $[-a,a]^n$ which are $a$-definable in the vector space structure.
Then $\operatorname{Th}(\mathcal V)$ has quantifier elimination in its language.
\end{proposition}
\begin{proof}
\cite[Proposition 5.1(1)]{LP}.
\end{proof}

The following theorem is a better variant of Lemma \ref{lem:in_omin}. 
It claims that an $\mathfrak X$-expansion of an ordered divisible abelian group always contains an o-minimal expansion of an ordered group.
\begin{theorem}\label{thm:in_omin}
Consider an $\mathfrak X$-expansion of an ordered divisible abelian group whose underlying set is $M$.
There exists an o-minimal expansion $\mathcal R$ of an ordered group having the same underlying set $M$ and satisfying the following conditions:
\begin{enumerate}
\item[(i)] Any set definable in $\mathcal R$ is $\mathfrak X$-definable.
\item[(ii)] Any bounded $\mathfrak X$-definable set is definable in $\mathcal R$.
\end{enumerate}
\end{theorem}
\begin{proof}
Since $(M,0,+,<)$ is an ordered divisible abelian group, it is naturally a $\mathbb Q$-vector space.
For any bounded $\mathfrak X$-definable set $X$, we define the predicate $P_X$ and interpret it naturally.
The notation $\mathcal S_{\text{bdd}}$ denote the set of all bounded $\mathfrak X$-sets. 
Set $\mathcal S_{\text{bdd}}(a)=\{X \in \mathcal S_{\text{bdd}}\;|\; X \subseteq [-a,a]^n \text{ when } X \in M^n\}$ for all positive $a \in M$.
Consider the language $L=(+,<,(a)_{a \in M}, (d)_{d \in \mathbb Q}, (P_X)_{X \in \mathcal S_{\text{bdd}}})$.
The set $M$ is naturally the underlying set of an $L$-structure $\mathcal R$.
The structure $\mathcal R$ obviously satisfies the condition (ii).
We demonstrate that $\mathcal R$ is an o-minimal structure satisfying the condition (i).

Take an arbitrary $L$-formula $\phi(\overline{x})$ with parameters in $M$.
Here, $\overline{x}$ is an $m$-tuple of variables.
We first show that the set defined by $\phi(\overline{x})$ is $\mathfrak X$-definable.
Since only finitely many symbols are involved in the $L$-formula $\phi(\overline{x})$,
we may assume that $\phi(\overline{x})$ is $L(a)$-formula with parameters for some $a>0$, where $L(a)=(+,<,a, (d)_{d \in \mathbb Q}, (P_X)_{X \in \mathcal S_{\text{bdd}}(a)})$.
There exists a quantifier-free formula $\psi(x)$ equivalent to $\phi(\overline{x})$ by Proposition \ref{prop:lp1}.
We may assume that $\phi(\overline{x})$ is a quantifier-free formula.
The formula $\phi(\overline{x})$ is a finite disjunction of a finite conjunction of  formulas of the forms $P_X(t(\overline{x},\overline{c}))$ and $t_1(\overline{x},\overline{c})\ *\ t_2(\overline{x},\overline{c})$ and their negations, where the notation $t(\overline{x},\overline{c})$ denotes an $n$-tuple of terms with parameters $\overline{c}$, $X$ is a subset of $M^n$ and $t_1(\overline{x},\overline{c})$ and $t_2(\overline{x},\overline{c})$ are terms with parameters $\overline{c}$.
The symbol $*$ is one of $=$, $<$ and $>$.

The terms are of the form $l(\overline{x})=\sum_{i=1}^n q_ix_i+c$, where $\overline{x}=(x_1,\ldots, x_n)$, $c \in M$ and $q_i \in \mathbb Q$ for all $1 \leq i \leq n$.
The definable set of the form $t_1(\overline{x},\overline{c})\ * \ t_2(\overline{x},\overline{c})$ is $\mathfrak X$-definable by Corollary \ref{cor:linear}.
Consider a formula of the form $P_X(t(\overline{x},\overline{c}))$.
Introduce new variables $\overline{y}=(y_1,\ldots, y_n)$.
The set defined by $$\{(\overline{x},\overline{y}) \in M^m \times M^n\;|\; \overline{y} \in X \text{ and } \overline{y}=t(\overline{x},\overline{c})\}$$ is $\mathfrak X$-definable because the graph of $t(\overline{x},\overline{c})$ is $\mathfrak X$-definable by Corollary \ref{cor:linear} and the intersection of $\mathfrak X$-definable sets is again $\mathfrak X$-definable by Definition \ref{def:x}(3).
The set defined by the formula $P_X(t(\overline{x},\overline{c}))$ is the image of the above $\mathfrak X$-definable set under the proper projection forgetting $\overline{y}$.
It is also $\mathfrak X$-definable by Definition \ref{def:x}(6).
The set defined by $\phi(\overline{x})$ is a boolean combination of sets of the above forms.
It is also $\mathfrak X$-definable because family of $\mathfrak X$-definable sets are closed under boolean algebra.
We have proven that the condition (i) is satisfied.

The next task is to demonstrate that the definable set defined by the formula $\phi(\overline{x})$ is a finite union of points and open intervals when $m=1$.
We may assume that $\phi(\overline{x})$ is a quantifier-free formula for the same reason as above.
A boolean combination of finite unions of points and open intervals is again a finite union of points and open intervals.
So we have only to demonstrate that the sets defined by formulas the forms $P_X(t(\overline{x},\overline{c}))$ and $t_1(\overline{x},\overline{c})* t_2(\overline{x},\overline{c})$ with $* \in \{=,<,>\}$ are finite unions of points and open intervals.
It is trivial in the latter case.
In the former case, the set defined by $P_X(t(\overline{x},\overline{c}))$ is a bounded $\mathfrak X$-definable set.
It is a finite union of points and open intervals by Definition \ref{def:x}(7).
\end{proof}

The following corollary indicates that the o-minimal structure $\mathcal R$ given in Theorem \ref{thm:in_omin} is the minimal $\mathfrak X$-structure and the $\mathfrak X$-structure of semi-definable sets in $\mathcal R$ is the maximal $\mathfrak X$-structure containing the o-minimal structure $\mathcal R$.
\begin{corollary}\label{cor:in_omin}
Consider an $\mathfrak X$-expansion $\mathcal X$ of an ordered divisible abelian group whose underlying set is $M$.
There exists an o-minimal expansion $\mathcal R$ of an ordered group whose underlying set is $M$ satisfying the following conditions:
\begin{enumerate}
\item[(i)] Any set definable in $\mathcal R$ is $\mathfrak X$-definable in $\mathcal X$.
\item[(ii)] Any set $\mathfrak X$-definable in $\mathcal X$ is $\mathfrak X$-definable in $\mathfrak X(R)$.
\end{enumerate}
Here, the notation $\mathfrak X(R)$ denotes the $\mathfrak X$-structure of semi-definable sets in $\mathcal R$.
\end{corollary}
\begin{proof}
Let $\mathcal R$ be the o-minimal structure given in Theorem \ref{thm:in_omin}.
The condition (i) follows from the theorem.
Take an arbitrary subset $X$ of $M^n$ $\mathfrak X$-definable in $\mathcal X$.
For any bounded open box $B$, the intersection $B \cap X$ is definable in $\mathcal R$.
It means that the set $X$ is semi-definable in $\mathcal R$.
We have demonstrated that the condition (ii) is satisfied.
\end{proof}

Here is another corollary.
Its proof illustrates a typical procedure for translating an assertion for o-minimal expansions of ordered groups into the corresponding assertion on bounded $\mathfrak X$-definable sets. 
\begin{corollary}[Curve selection lemma]\label{cor:curve_selection}
Consider an $\mathfrak X$-expansion $\mathcal X$ of an ordered divisible abelian group whose underlying set is $M$.
Let $X$ be an $\mathfrak X$-definable subset of $M^n$ and take a point $a \in \partial X$.
Let $\mathcal R$ be the o-minimal structure given in Theorem \ref{thm:in_omin}.
There exist a positive $\varepsilon \in M$ and a continuous map $\gamma:]0,\varepsilon[ \rightarrow X$ definable in $\mathcal R$ such that the image of $\gamma$ is bounded and $\lim_{ t \to 0}\gamma(t)=a$.  
\end{corollary}
\begin{proof}
Let $\mathcal R$ be the o-minimal structure given in Theorem \ref{thm:in_omin}.
Take a bounded open box $U$ containing the point $a$.
The set $X \cap U$ is definable in $\mathcal R$ by Theorem \ref{thm:in_omin}.
Since $\mathcal R$ is an o-minimal expansion of an ordered group, there exists a continuous map $\gamma:]0,\varepsilon[ \rightarrow X \cap U$ definable in $\mathcal R$ such that $\lim_{ t \to 0}\gamma(t)=a$ by the curve selection lemma for o-minimal expansions of ordered groups \cite[Chapter 6, Corollary 1.5]{vdD}.
The image of $\gamma$ is bounded because it is contained in $U$.
\end{proof}

\subsection{Dimension}
We define the dimension of an $\mathfrak X$-definable set and investigate its basic properties.

\begin{definition}\label{def:dimension}
Consider a densely linearly ordered set without endpoints $(M,<)$.
Let $X$ be a subset of $M^n$.
If $X$ is an empty set, we set $\dim X=-\infty$.
The nonempty set $X$ is of dimension $\geq m$ if there exist a point $x \in M^n$ and a coordinate projection $\pi:M^n \rightarrow M^m$ such that, for any open box $B$ containing the point $x$, the projection image $\pi(B \cap X)$ has a nonempty interior.
The set is of dimension $m$ when it is of dimension $\geq m$ and not of dimension of $\geq m+1$.
\end{definition}

\begin{remark}
The dimensions of sets definable in an o-minimal structure and definable in a locally o-minimal admitting local definable cell decomposition are defined differently in \cite[Chapter 4, (1.1)]{vdD} and \cite[Definition 5.1]{Fuji}, respectively.
However, they coincide with the definition given above by \cite[Corollary 5.3]{Fuji}.
We use this fact without any notification in the rest of this paper.
\end{remark}

\begin{lemma}\label{lem:dim0}
Consider an $\mathfrak X$-structure.
A nonempty $\mathfrak X$-definable set is of dimension zero if and only if it is discrete.
A discrete $\mathfrak X$-definable set is closed.
\end{lemma}
\begin{proof}
Lemma \ref{lem:in_omin} and the cell decomposition theorem \cite[Chapter 3, Theorem 2.11]{vdD} immediately imply this lemma.
\end{proof}

The following lemma is well-known:
\begin{lemma}\label{lem:omin_tmp}
Consider an o-minimal structure whose underlying set is $M$.
Let $C_1, \ldots, C_N$ be definable subsets of $M^n$.
If the union $\bigcup_{i=1}^N C_i$ has a nonempty interior, $C_k$ has a nonempty interior for some $1 \leq k \leq N$. 
\end{lemma}
\begin{proof}
It is an easy corollary of the cell decomposition theorem \cite[Chapter 3, Theorem 2.11]{vdD}.
\end{proof}

We give another expression of dimension.
\begin{lemma}\label{lem:equiv_dim}
Consider an $\mathfrak X$-structure whose underlying set is $M$.
Let $X$ be an $\mathfrak X$-definable subset of $M^n$.
We have $$\dim X = \sup\{ \dim (U \cap X) \;|\; U \text{ is a bounded open box in }M^n\}\text{.}$$
\end{lemma}
\begin{proof}
Let $\mathcal R$ be the o-minimal structure given in Lemma \ref{lem:in_omin}.
Let $d$ be the right hand of the equality in the lemma.
We first demonstrate that $d \leq \dim X$.
There exists a bounded open box $U$ in $M^n$ such that $d=\dim (U \cap X)$.
The set $U \cap X$ is definable in $\mathcal R$.
There exists a cell $C$ contained in $U \cap X$ and a coordinate projection $\pi:M^n \rightarrow M^d$ such that the image $\pi(C)$ has a nonempty interior by the definition of dimension of a set definable in the o-minimal structure $\mathcal R$.

Take an arbitrary point $x \in C$.
The projection image $\pi(C \cap B)$ has a nonempty interior for any open box $B$ containing the point $x$ by the definition of cells.
 Since $C$ is a subset of $X$, the projection image $\pi(X \cap B)$ has a nonempty interior for any open box $B$ containing the point $x$.
 It means that $d \leq \dim X$.
 
We next demonstrate the opposite inequality $\dim X \leq d$.
There exist a point $x \in M$ and a coordinate projection $\pi:M^n \rightarrow M^{\dim X}$ such that, for any open box $B$ containing the point $x$, the projection image $\pi(B \cap X)$ has a nonempty interior.
 Fix a bounded open box $B$ containing the point $x$.
 The intersection $X \cap B$ is definable in $\mathcal R$ by Lemma \ref{lem:in_omin}.
 Apply the cell decomposition theorem \cite[Chapter 3, Theorem 2.11]{vdD}.
 We get a finite partition into cells $X \cap B=\bigcup_{i=1}^N C_i$.
 The image $\pi(C_k)$ has a nonempty interior for some $1 \leq k \leq N$ by Lemma \ref{lem:omin_tmp}.
 It implies that $\dim X \leq \dim C_k \leq \dim (B \cap X)$.
 It means $\dim X \leq d$.
\end{proof}

We summarize the basic properties of dimension.
\begin{proposition}\label{prop:dim}
Consider an $\mathfrak X$-structure whose underlying set is $M$.
The following assertions hold true:
\begin{enumerate}
\item[(a)] We have $\dim (X) \leq \dim (Y)$ for any $\mathfrak X$-definable sets $X$ and $Y$ with $X \subseteq Y$.
\item[(b)] The equality $\dim (X \cup Y)=\max\{\dim(X),\dim(Y)\}$ holds true for any $\mathfrak X$-definable subsets $X$ and $Y$ of $M^n$.
\item[(c)] The equality $\dim (X \times Y)=\dim(X)+\dim(Y)$ holds true for any $\mathfrak X$-definable sets $X$ and $Y$.
\item[(d)] Let $X$ be an $\mathfrak X$-definable set.
We get $\dim (\partial X) < \dim(X)$ and $\dim(\overline{X})=\dim X$ when $\partial X$ is $\mathfrak X$-definable.
\end{enumerate}
\end{proposition}
\begin{proof}
This proposition immediately follows from Lemma \ref{lem:equiv_dim} and the basic properties of the dimension of sets definable in an o-minimal structure \cite[Chapter 4, Proposition 1.3, Corollary 1.6, Theorem 1.8]{vdD}.
\end{proof}

\subsection{Structure theorem}
Loveys and Peterzil investigated necessary and sufficient conditions for an o-minimal expansion of an ordered group being linear in \cite{LP}.
Using their results, we investigate the structure of an $\mathfrak X$-expansion $\mathcal X$ of an ordered divisible abelian group when there exists an $\mathfrak X$-definable strictly monotone homeomorphism between a bounded open interval and an unbounded open interval. 
We first prove the following lemma.

\begin{lemma}\label{lem:xbij}
Consider an $\mathfrak X$-expansion of an ordered divisible abelian group whose underlying set is $M$.
Assume that there exists an $\mathfrak X$-definable strictly monotone homeomorphism between a bounded open interval and an unbounded open interval. 
Then, there exists an $\mathfrak X$-definable strictly increasing homeomorphism $\psi$ between an arbitrary bounded open interval and $M$ such that $$\psi(\text{the middle point of the bounded open interval})=0\text{.}$$
\end{lemma}
\begin{proof}
Let $\varphi:I \rightarrow J$ be the given $\mathfrak X$-definable strictly monotone homeomorphism between a bounded open interval $I$ and an unbounded open interval $J$.
It is easy to construct strictly increasing homeomorphisms between $]a,b[$ and $]0,b-a[$, between $]a,\infty[$ and $]0,\infty[$, and between $]-\infty,a[$ and $]-\infty,0[$ for $a,b \in M$ using the addition.
The composition of $\varphi$ with them is $\mathfrak X$-definable by Lemma \ref{lem:composition} and Lemma \ref{lem:bounded}(1).
Hence, we may assume that $I=]0,u[$ for some $u>0$.
If $J=M$, consider the restriction of $\varphi$ to $]0,\varphi^{-1}(0)[$ instead of $\varphi$.
We may assume that $J=]-\infty,0[$ or $J=]0,\infty[$.
We have to check $\mathfrak X$-definability in the same manner as above every time we construct a new map but we omit them in the proof.

We may further assume that $J=]0,\infty[$ and $\varphi$ is strictly increasing by composing $\varphi$ with the maps given by $x \mapsto -x$ and $x \mapsto u-x$.
Take an arbitrary nonempty bounded open interval $I'$.
We will construct an $\mathfrak X$-definable strictly increasing homeomorphism from $I'$ to $M$.
We may assume that $I'$ is of the form $]0,v[$ for some $v>0$ in the same way as above.
We first construct an $\mathfrak X$-definable strictly increasing homeomorphism from $I'$ to $]0,\infty[$.
We have nothing to do when $u=v$.
When $v<u$, the $\mathfrak X$-map given by $\varphi(t+u-v)-\varphi(u-v)$ for all $0 <t <v$ is the desired map.
When $v>u$, consider the the $\mathfrak X$-map which is the identity map on $]0,v-u]$ and which is given by $\varphi(t+u-v)+v-u$ for all $t > v-u$.

We finally construct an $\mathfrak X$-definable strictly increasing homeomorphism from $I'$ to $M$.
We can construct $\mathfrak X$-definable strictly increasing homeomorphisms $\varphi_1:]0,v/2[ \rightarrow ]-\infty,0[$ and $\varphi_1:]v/2,v[ \rightarrow ]0, \infty[$ in the same manner as above.
The map $\psi:]0,v[ \rightarrow M$ given by $\psi(t)=\varphi_1(t)$ if $t<v/2$, $\psi(v/2)=0$ and $\psi(t)=\varphi_2(t)$ if $t>v/2$ is the desired map.
\end{proof}

Recall the definition of a piecewise linear map definable in an o-minimal expansion of an ordered group.
We also define a piecewise linear map definable in an $\mathfrak X$-expansion of an ordered divisible abelian group.
\begin{definition}
Consider an o-minimal expansion of an ordered group $\mathcal M=(M,<,+,0,\ldots)$.
A definable function $F:U \subseteq M^n \rightarrow M$ is \textit{piecewise linear}, if we can partition $U$ into finitely many definable sets $U_1, \ldots, U_k$ such that $F$ is linear on each of them, i.e., given $x,y \in U_i$ and $t \in M^n$, if $x+t,y+t \in U_i$, then $F(x+t)-F(x)=F(y+t)-F(y)$.

Consider an $\mathfrak X$-expansion of an ordered divisible abelian group whose underlying set is $M$.
An $\mathfrak X$-definable function $F:U \subseteq M^n \rightarrow M$ is \textit{piecewise linear}, for any bounded open box $B$ in $M^n$ and a bounded open interval $I$, if we can partition $B \cap U \cap F^{-1}(I)$ into finitely many $\mathfrak X$-definable sets $U_1, \ldots, U_k$ such that $F$ is linear on each of them, i.e., given $x,y \in U_i$ and $t \in M^n$, if $x+t,y+t \in U_i$, then $F(x+t)-F(x)=F(y+t)-F(y)$.
\end{definition}

The following is due to Loveys and Peterzil \cite{LP} and is summarized in \cite{E}.
See also \cite{PS2}.
\begin{proposition}\label{prop:lp2}
Consider an o-minimal expansion of an ordered group $\mathcal M=(M,<,+,0,\ldots)$.
The following are equivalent:
\begin{enumerate}
\item[(1)] Every definable function $F:U \subseteq M^n \rightarrow M$ is piecewise linear.
\item[(2)] There exist no definable binary operations $\oplus, \otimes:I^2 \rightarrow I$ on an interval $I=]-a,a[$, and a positive element $1 \in I$ such that $(I,<_I,0,1,\oplus,\otimes)$ is an ordered real closed field, where $<_I$ denotes the restriction of $<$ to $I$.
\end{enumerate}
\end{proposition}
\begin{proof}
\cite[Fact 1.12]{E}.
\end{proof}

We are now ready to prove the structure theorem.
\begin{theorem}\label{thm:xstr}
Consider an $\mathfrak X$-expansion of an ordered divisible abelian group whose underlying set is $M$.
Assume further that there exists an $\mathfrak X$-definable strictly monotone homeomorphism between a bounded open interval and an unbounded open interval. 
Then, exactly one of the following holds true:
\begin{enumerate}
\item[(1)] Any $\mathfrak X$-definable function is piecewise linear.
\item[(2)] The structure is an $\mathfrak X$-expansion of an ordered real closed field in the following sense:
There exists elements $1' \in M$ and $\mathfrak X$-definable binary operations $\oplus,\otimes:M^2 \rightarrow M$ such that the tuple $(M,<,0,1',\oplus,\otimes)$ is an ordered real closed field.
\end{enumerate}
\end{theorem}
\begin{proof}
Take the o-minimal expansion of an ordered group $\mathcal R$ given in Theorem \ref{thm:in_omin}.
We have the following two cases by Proposition \ref{prop:lp2}.
\begin{enumerate}
\item[(1)] Every function $F:U \subseteq M^n \rightarrow M$ definable in $\mathcal R$ is piecewise linear.
\item[(2)] There exist binary operations $\oplus_I, \otimes_I:I^2 \rightarrow I$ definable in $\mathcal R$ on an interval $I=]-a,a[$, and a positive element $1_I \in I$ such that $(I,<_I,0,1_I,\oplus_I,\otimes_I)$ is an ordered real closed field, where $<_I$ denotes the restriction of $<$ to $I$.
\end{enumerate}

We first consider the case (1).
Take an arbitrary $\mathfrak X$-function $F:U \subseteq M^n \rightarrow M$.
Take a bounded open box $B$ in $M^n$ and a bounded open interval $I$.
The set $\Gamma(F) \cap (B \times I)$ is definable in $\mathcal R$, where $\Gamma(F)$ denotes the graph of the function $F$.
It is the graph of the restriction of $F$ to $U \cap B \cap F^{-1}(I)$.
Since the function $F|_{U \cap B \cap F^{-1}(I)}$ definable in $\mathcal R$ is piecewise linear, the $\mathfrak X$-definable function is also piecewise liner.

We next treat the case (2).
There exists a strictly increasing $\mathfrak X$-definable homeomorphism $\varphi:I \rightarrow M$ with $\varphi(0)=0$ by Lemma \ref{lem:xbij}.
Set $1'=\varphi(1_I)$, $x \oplus y= \varphi(\varphi^{-1}(x) \oplus_I \varphi^{-1}(y))$ and $x \otimes y= \varphi(\varphi^{-1}(x) \otimes_I \varphi^{-1}(y))$ for all $x, y \in M$.
The graph of $\oplus_I$ is a bounded set definable in $\mathcal R$.
In particular, it is $\mathfrak X$-definable.
The graph of $\oplus$ is the image of the graph of $\oplus_I$ under the homeomorphism between $I^3$ and $M^3$ given by $(x,y,z) \mapsto (\varphi(x),\varphi(y),\varphi(z))$.
It is $\mathcal X$-definable by Lemma \ref{lem:image}.
The operator $\otimes$ is also $\mathfrak X$-definable for the same reason.
It is easy to check that the tuple $(M,<,0,1',\oplus,\otimes)$ is an ordered real closed field using \cite[Theorem 1.2.2(ii)]{BCR}.
We omit the details.
\end{proof}

\begin{remark}
Consider an o-minimal expansion $\widetilde{\mathbb R}$ of the ordered group of reals.
The $\mathfrak X$-structure of semi-definable sets in $\widetilde{\mathbb R}$ satisfies the assumption of Theorem \ref{thm:xstr}.
The map $\varphi:]0,1[ \rightarrow ]0,\infty[$ defined by $\varphi(x)=i+2^{i+1}(x-(1-1/2^i))$ for $1-1/2^i<x \leq 1-1/2^{i+1}$ is a semi-definable homeomorphism between the interval $]0,1[$ and the interval $]0,\infty[$.
\end{remark}

\begin{remark}
An assertion for Shiota's $\mathfrak Y$-sets similar to but not identical to Theorem \ref{thm:xstr} is found in \cite[Theorem V.2.2]{Shiota}. 
\end{remark}

\subsection{Topological results}
We summarize other basic topological properties of $\mathfrak X$-structures which are not introduced in the previous subsections.
\begin{proposition}\label{prop:x_int_closure}
Consider an $\mathfrak X$-expansion of an ordered divisible abelian group.
The interior, closure and frontier of an $\mathfrak X$-definable set are $\mathfrak X$-definable.
\end{proposition}
\begin{proof}
Let $M$ be the underlying space of the given $\mathfrak X$-structure.
Let $X$ be an $\mathfrak X$-definable subset of $M^n$.
Fix a positive $r>0$.
Consider the $\mathfrak X$-definable set 
$$ A=\{(x,y,s) \in M^n \times M^n \times M\;|\; 0<s<r,\ |x_i - y_i|<s (\forall i),\  x \in X,\  y \not\in X\} \text{,}$$
where $x_i$ and $y_i$ are the $i$-th coordinate of $x$ and $y$, respectively.
Set $$B=\{(x,s) \in X \times M\;|\; 0<s<r,\ \exists y \not\in X \text{ and }  |x_i - y_i|<s (\forall i)\}\text{.}$$
The set $B$ is the image of the $\mathfrak X$-definable set $A$ under a proper projection.
The set $B$ is $\mathfrak X$-definable.
The interior $\myint(X)$ of $X$ is the image of $(X \times \{s \in M\;|\;0<s<r\}) \setminus B$ under the projection forgetting the last coordinate.
Therefore,  the interior is $\mathfrak X$-definable.

The closure of $X$ is given by $(\myint(X^c))^c$.
Here, the notation $A^c$ denotes the complement of a set $A$.
The closure is $\mathfrak X$-definable.
The frontier is also $\mathfrak X$-definable by Definition \ref{def:x}(3).
\end{proof}

\begin{corollary}\label{cor:dim}
Consider an $\mathfrak X$-expansion of an ordered divisible abelian group.
Let $X$ be an $\mathfrak X$-definable set.
We get $\dim (\partial X) < \dim(X)$ and $\dim(\overline{X})=\dim X$.
\end{corollary}
\begin{proof}
Immediate from Proposition \ref{prop:x_int_closure} and Proposition \ref{prop:dim}(d).
\end{proof}

\begin{example}\label{ex:x1}
Consider a definably complete structure $\mathcal M=(M,<,\ldots)$.
A \textit{definable family} of subsets of $M^n$, parameterized by $A \subseteq M^m$, is an indexed family $\{Y_a\}_{a \in A}$ of fibers, where $Y \subseteq M^{m+n}$ and $A \subseteq M^m$ are definable.
A definable family $\{Y_a\}_{a \in A}$ is \textit{monotone} if $A \subseteq M$ and either $Y_r \supseteq Y_s$ for all $r,s \in A$ with $r \leq s$ or $Y_r \subseteq Y_s$ for all $r,s \in A$ with $r \leq s$.
We have $\bigcap_{r \in A} Y_r \neq \emptyset$ for all monotone definable families of nonempty definable closed and bounded sets $\{Y_r\}_{r \in A}$ by \cite[Lemma 1.9]{M}.

We can define $\mathfrak X$-definable family similarly.
But the intersection of monotone $\mathfrak X$-definable family of nonempty $\mathfrak X$-definable closed and bounded sets may be an empty set.
In fact, consider the o-minimal structure in Example \ref{ex:x2} and the $\mathfrak X$-structure of semi-definable sets in this o-minimal structure.

The set of positive integers $\mathbb N$ is semi-definable.
Take $a_n,b_n \in \mathbb Q$ as in Example \ref{ex:x2}.
Set $Y_n = [a_n,b_n] \subseteq \mathbb R_{\text{alg}}$.
The family $\{Y_n\}_{n \in \mathbb N}$ is a monotone $\mathfrak X$-definable family of nonempty definable closed and bounded sets.
We obviously have $\bigcap_{n \in \mathbb N} Y_n = \emptyset$.
\end{example}

%

We finally study when there exists an unbounded discrete $\mathfrak X$-definable set in $M$.
\begin{lemma}\label{lem:xoromin}
Consider an $\mathfrak X$-expansion of an ordered divisible abelian group whose underlying set is $M$.
Exactly one of the following conditions holds true:
\begin{enumerate}
\item[(1)] Any $\mathfrak X$-definable subset of $M$ is a finite union of points and open interval.
\item[(2)] There exists an unbounded discrete $\mathfrak X$-definable set.
\end{enumerate}
\end{lemma}
\begin{proof}
Assume that the condition (1) is not satisfied.
There exists an $\mathfrak X$-definable subset $X$ of $M$ which is not a finite union of points and open intervals.
Set $Y=\overline{X} \setminus \myint(X)$.
It is $\mathfrak X$-definable by Proposition \ref{prop:x_int_closure}.
We demonstrate that $Y$ is an unbounded discrete $\mathfrak X$-definable set.

We first demonstrate that $Y$ is discrete.
For any bounded open interval $I$ in $M$, the intersection $I \cap Y$ is a finite union of points and open intervals by the definition of a $\mathfrak X$-structure.
Therefore, $Y$ is discrete when it has an empty interior.
Assume that $Y$ has a nonempty interior.
We can take a bounded open interval $J$ contained in $Y$.
Since $X \cap J$ is a finite union of points and open intervals, $J = Y \cap J = (\overline{X} \cap J) \setminus (\myint(X) \cap J)$ consists of finite points.
Contradiction.

We next show that $Y$ is unbounded.
The set $Y$ consists of infinite points.
In fact, assume that $Y$ is a finite set.
There exists a nonempty bounded open interval $I$ which contains $Y$.
The difference $M \setminus I$ consists of a closed interval $J_+$ unbounded above and a closed interval $J_-$ unbounded below.
We have $J_+ \cap X = \emptyset$ or $J_+ \subseteq X$.
Otherwise, we can take points $a \in J_+ \cap X$ and $b \in J_+ \setminus X$.
Take a bounded open interval $J \subseteq J_+$ containing the points $a$ and $b$.
We have $J \cap Y=\emptyset$.
The set $J \cap X$ is a finite union of points and open intervals.
We have $J \cap X \neq \emptyset$ because it contains the point $a$.
We also have $J \cap X \neq J$ because $J \cap X$ does not contain the point $b$.
By the definition of $Y$, we get $Y \cap J \neq \emptyset$ in both the cases in which $X \cap J$ consists of points and in which it contains an open interval.
We have demonstrated that $J_+ \cap X = \emptyset$ or $J_+ \subseteq X$.
We also obtain $J_- \cap X = \emptyset$ or $J_- \subseteq X$ similarly.
Since $I \cap X$ is a finite union of points and open intervals, the set $X$ is a finite union of points and open intervals.
Contradiction.
We have demonstrated that $Y$ consists of infinite points.

If $Y$ is bounded, we can take a bounded open interval $I$ containing the set $Y$.
The set $Y=Y \cap I$ consists of finite points because $Y$ is $\mathfrak X$-definable and $Y$ is discrete.
Contradiction to the fact $Y$ is an infinite set.
\end{proof}

\section{Geometry of semi-definable sets}\label{sec:semi-definable}
We studied $\mathfrak X$-structures in the previous section.
In this section, we treat a special family of $\mathfrak X$-structures; that is, the $\mathfrak X$-structure of semi-definable sets in an o-minimal structure $\mathcal R=(M,<,\ldots)$.
\subsection{Frontier of semi-definable set}
We first consider the frontier, interior and closure of semi-definable sets.
\begin{lemma}\label{lem:frontier}
Consider an o-minimal structure.
The frontier, interior and closure of a semi-definable set are semi-definable.
\end{lemma}
\begin{proof}
Let $\mathcal R$ be an o-minimal structure and $M$ be its underlying set.
Let $X$ be a semi-definable subset of $M^n$.
We have $(\partial X) \cap U = (\partial (X \cap U)) \cap U$ for any bounded open box $U$.
The set $X \cap U$ is definable in $\mathcal R$ because $X$ is semi-definable.
The frontier $\partial (X \cap U)$ is also definable.
The intersection $(\partial X) \cap U $ is definable.
It means that $\partial X$ is semi-definable.

Once we know that the frontier is semi-definable, it is easy to demonstrate that the interior and the closure are semi-definable.
\end{proof}

\begin{remark}
When we assume that the o-minimal structure is an expansion of an ordered group, Lemma \ref{lem:frontier} immediately follows from Proposition \ref{prop:x_int_closure}.  
\end{remark}

\subsection{Semi-definable connectedness}
We next introduce the notion of semi-definable connectedness.
\begin{definition}
Consider an o-minimal structure $\mathcal R=(M,<,\ldots)$. 
A semi-definable subset $X$ of $M^n$ is \textit{semi-definably connected} if there are no non-empty proper semi-definable closed and open subsets $Y_1$ and $Y_2$ of $X$ such that $Y_1 \cap Y_2 = \emptyset$ and $X=Y_1 \cup Y_2$.
We define that a definable set is \textit{definably connected} in the same manner.
The semi-definable set $X$ is \textit{semi-definably pathwise connected} if, for any $x,y \in X$, there exist elements $c_1,c_2 \in M$ and a definable continuous map $\gamma:[c_1,c_2] \rightarrow X$ with $\gamma(c_1)=x$ and $\gamma(c_2)=y$.
We define that a definable set is \textit{definably pathwise connected} in the same manner.
\end{definition}

We easily get the following result:
\begin{lemma}[Intermediate value property]\label{lem:intermediate}
Consider an o-minimal structure $\mathcal M=(M,<,\ldots)$.
Let $D$ be a subset of $M^n$ and $f:D \rightarrow M$ be a function whose graph is semi-definable and semi-definably connected.
Take two points $y_1, y_2 \in f(D)$ with $y_1<y_2$.
For any $y \in M$ with $y_1<y<y_2$, there exists $x \in D$ such that $y=f(x)$.
\end{lemma}
\begin{proof}
Otherwise, the sets $\Gamma(f) \cap (M^n \times \{y' \in M\;|\; y'>y\})$ and $\Gamma(f) \cap (M^n \times \{y' \in M\;|\; y'<y\})$ are nonempty closed and open semi-definable subsets of $\Gamma(f)$.
Here, $\Gamma(f)$ denotes the graph of $f$.
\end{proof}

We next recall the following fact:
\begin{lemma}\label{lem:connected_omin}
Consider an o-minimal structure $\mathcal R=(M,<,\ldots)$. 
Let $X$ be a definable subset of $M^n$ and $U_1 \subseteq U_2$ be open boxes in $M^n$.
Take a definably connected component $C$ of $X \cap U_2$.
The intersection $C \cap U_1$ is the union of the definably connected components of $X \cap U_1$ contained in $C$.
\end{lemma}
\begin{proof}
Immediate from the definable cell decomposition theorem for o-minimal structures \cite[Chapter 3, Theorem 2.11]{vdD}.
\end{proof}

We get the following theorem:
\begin{theorem}\label{thm:connected}
Consider an o-minimal expansion $\mathcal R=(M, <, +, 0, \ldots)$ of an ordered group.
Let $X$ be a nonempty semi-definable subset of $M^n$.
The following are equivalent:
\begin{enumerate}
\item[(1)] $X$ is semi-definably connected.
\item[(2)] For any $x,y \in X$, there exists a bounded open box $U$ in $M^n$ such that both the points $x$ and $y$ are contained in some definably connected component of $X \cap U$.
\item[(3)] $X$ is semi-definably pathwise connected.
\end{enumerate}
In addition, for any $x \in X$, there exists a maximal semi-definably connected semi-definable subset $Y$ of $X$ containing the point $x$.
The set $Y$ is called the semi-definably connected component of $X$ containing the point $x$.
A semi-definably connected component of $X$ is closed and open in $X$.
\end{theorem}
\begin{proof}
Fix a point $x \in X$.
We first define a semi-definable closed and open subset $C_x$ of $X$ containing the point $x$.
In this proof, the subscript such as $C_x$ does not denote the fiber of a set, exceptionally.
Let $\mathcal B_x$ be the set of bounded open boxes in $M^n$ containing the point $x$. 
Take an arbitrary element $U \in \mathcal B_x$.
The intersection $U \cap X$ is definable by the definition of semi-definablity and it has finite definably connected components by \cite[Chapter 3, Proposition 2.18]{vdD}.
Let $\mathcal C(U)$ be the set of the definably connected components of $U \cap X$.
The notation $\mathcal L_x(U)$ denotes the subset of $\mathcal C(U)$ of the elements $C$ satisfying that a definably connected component of the intersection $B \cap X$ contains both the point $x$ and $C$ for some bounded open box $B$ containing the open box $U$.
We set $$C_x(U)=\bigcup_{C \in L_x(U)}C\text{.}$$

Let $U_1, U_2 \in \mathcal B_x$ with $U_1 \subseteq U_2$. 
We show the following equality:
$$
C_x(U_2) \cap U_1 = C_x(U_1)\text{.}
$$
Since $C_x(U_2)$ is a finite union of definably connected components of $X \cap U_2$, the set $C_x(U_2) \cap U_1$ is also a finite union of definably connected components of $X \cap U_1$ by Lemma \ref{lem:connected_omin}.
Let $D_1$ be a definably connected component of $X \cap U_1$.
The exists a unique definably connected component $D_2$ of $X \cap U_2$ containing the set $D_1$ by Lemma \ref{lem:connected_omin}.
We have only to demonstrate that $D_1 \in \mathcal L_x(U_1)$ if and only if $D_2 \in \mathcal L_x(U_2)$.
We can easily demonstrate that $D_1 \in \mathcal L_x(U_1)$ when $D_2 \in \mathcal L_x(U_2)$.
We omit the proof.
We consider the opposite implication.
There exists a bounded open box $U$ such that $D_1 \subseteq U$ and a definably connected component of $X \cap U$ contains both $x$ and $D_1$.
Taking a larger bounded open box $V$ containing the both $U$ and $U_2$.
A definably connected component of $X \cap V$ still contains both $x$ and $D_1$.
The definably connected set $D_2$ is also contained in the same definably connected component of $X \cap V$ by Lemma \ref{lem:connected_omin} because $D_2$ contains $D_1$ by the assumption.
It means that $D_2 \in \mathcal L_x(U_2)$.

We are now ready to define the semi-definable closed and open subset $C_x$.
Set $$C_x=\bigcup_{U \in \mathcal B_x}C_x(U)\text{.}$$
We first show that $C_x \cap U = C_x(U)$ for any $U \in \mathcal B_x$.
In fact, the inclusion $C_x(U) \subseteq C_x \cap U$ is obvious from the definition.
We demonstrate the opposite inclusion.
Take an arbitrary element $y \in C_x \cap U$.
There exists $V \in \mathcal B_x$ such that $y \in C_x(V)$ by the definition.
Set $W=U \cap V$.
Since $C_x(V)$ is a subset of $V$, we get $y \in C_x(V) \cap U = C_x(V) \cap (V \cap U) = C_x(V) \cap W = C_x(W) = C_x(U) \cap W \subseteq C_x(U)$ by the above equality. 
We have demonstrated the opposite inclusion.

We demonstrate that $C_x$ is semi-definable.
Take an arbitrary bounded open box $U$ in $M^n$.
We have only to prove that $C_x \cap U$ is definable.
Take a bounded open box $V$ larger than $U$.
If $C_x \cap V$ is definable then $C_x \cap U = (C_x \cap V) \cap U$ is also definable.
We may assume that $U$ contains the point $x$ for the above reason.
We get $C_x \cap U=C_x(U)$, which is a finite union of definably connected components of $X \cap U$.
Hence, it is definable.

The semi-definable set $C_x$ is closed and open in $X$.
In fact, take a point $y \in C_x$.
Take a bounded open box $U$ containing the points $x$ and $y$.
The intersection $C_x \cap U=C_x(U)$ is a finite union of definably connected components of $X \cap U$.
In particular, $C_x \cap U$ is closed and open in $X \cap U$ by \cite[Chapter 3, Proposition 2.18]{vdD}.
Take a sufficiently small open box $V \subseteq U$ containing the point $y$.
We have $X \cap V =(X \cap U) \cap V=(C_x \cap U) \cap V = C_x \cap V$ by the definition of definably connected components.
It means that $C_x$ is open in $X$.
We can demonstrate that $X \setminus C_x$ is open in the same manner. 

We have finished the preparation.
We prove that (1) implies (2).
Assume that the condition (2) does not hold true.
There exist $x,y \in X$ such that, for any bounded open box $U$ containing the point $x$ and $y$, $x$ and $y$ are contained in different definably connected components of $X \cap U$.
It means that $\mathcal L_x(U)$ and $\mathcal L_y(U)$ has an empty intersection for any bounded open box $U$ containing $x$ and $y$.
We have $C_x(U) \cap C_y(U) = \emptyset$ by the definition for any bounded open box $U$ containing both $x$ and $y$.
We easily get $C_x \cap C_y = \emptyset$.
It means that $X$ is not semi-definably connected. 

The next task is to prove that $(2) \Rightarrow (3)$.
Take arbitrary $x,y \in X$.
There exist a bounded open box $U$ containing the points $x$ and $y$ and a definably connected component $Y$ of $X \cap U$ containing the points $x$ and $y$.
It is well-known that a definable connected set is definably pathwise connected \cite[Chapter 6, Proposition 3.2]{vdD}.
There exists a definable continuous map $\gamma:[c_1,c_2] \rightarrow X \cap U$ with $\gamma(c_1)=x$ and $\gamma(c_2)=y$.

The implication $(3) \Rightarrow (1)$ is easy to be proven.
Assume for contradiction that there exist disjoint nonempty semi-definable closed and open subsets $Y_1$ and $Y_2$ of $X$ with $X=Y_1 \cup Y_2$.
Take points $y_1, y_2 \in X$ with $y_i \in Y_i$ for $i=1,2$.
There exists a definable continuous map $\gamma:[c_1, c_2] \rightarrow X$ with $\gamma(c_i)=y_i$ for $i=1,2$.
The image of $\gamma$ is bounded by \cite[Proposition 1.10]{M} because an o-minimal structure is definably complete.
We can take a bounded open box $B$ in $M^n$ containing the image of $\gamma$.
The sets $B \cap X$, $B \cap Y_1$ and $B \cap Y_2$ are all definable.
The closed interval $[c_1,c_2]$ is decomposed into two disjoint definable closed and open subsets $\gamma^{-1}(Y_1 \cap B)$ and  $\gamma^{-1}(Y_2 \cap B)$.
On the other hand, the closed interval is definably connected by \cite[Corollary 1.5]{M}.
It is a contradiction.

The last task is to prove the existence of a semi-definably connected component.
In fact, the semi-definable set $C_x$ is the semi-definably connected component containing the point $x \in X$.
Take arbitrary $y_1, y_2 \in C_x$.
We can take $U_1, U_2 \in \mathcal B_x$ with $y_i \in C_x(U_i)$ for $i=1,2$.
By the definition of $C_x(U_i)$, there exists $V_i \in \mathcal B_x$ such that $x$ and $y_i$ are contained in a definably connected component of $V_i \cap X$.
Take a bounded open box $W$ containing the open boxes $V_1$ and $V_2$.
Three points $x$, $y_1$ and $y_2$ are contained in a definably connected component of $X \cap W$.
This definably connected component is also the definably connected component of $C_x(W)=C_x \cap W$.
Hence, $C_x$ is semi-definably connected by the condition (2). 

We finally show that $C_x$ is maximal.
Take an arbitrary semi-definably connected semi-definable subset $Y$ of $X$ with $x \in Y$.
We have only to demonstrate that $Y$ is contained in $C_x$.
Take an arbitrary point $y \in Y$. 
Since $Y$ is semi-definably connected, there exists a definable continuous map $\gamma:[c_1,c_2] \rightarrow Y$ such that $\gamma(c_1)=x$ and $\gamma(c_2)=y$.
Since the image $\gamma([c_1,c_2])$ is bounded for the same reason as above, we can take a bounded open box $U$ containing the image.
It means that $x$ and $y$ are contained in the same definably connected component of $X \cap U$ because a definably pathwise connected definable set is definably connected.
We have $y \in C_x(U) \subseteq C_x$.
We have finished the proof.
\end{proof}

We introduce a corollary of Theorem \ref{thm:connected}.

\begin{corollary}\label{cor:connected}
Consider an o-minimal expansion of an ordered group.
The closure of a nonempty semi-definably connected semi-definable set is again semi-definably connected.
\end{corollary}
\begin{proof}
Immediately follows from Theorem \ref{thm:connected}(3) and Corollary \ref{cor:curve_selection}.
\end{proof}

\subsection{Good manifolds}

We introduce the notion of a good manifold necessary in Section \ref{sec:multi}.

\begin{definition}\label{def:manifold}
Consider an o-minimal structure $\mathcal R=(M,<,\ldots)$.
A subset $X$ of $M^n$ of dimension $d$ is \textit{locally a good submanifold} at $x \in X$ if there exist 
\begin{itemize}
\item a bounded open box $B$ containing the point $x$, 
\item a permutation $\sigma$ of $\{1,\ldots, n\}$ and 
\item a definable continuous map $f: \pi_d(\widetilde{\sigma}(X \cap B)) \rightarrow M^{n-d}$ 
\end{itemize}
such that $\widetilde{\sigma}(X \cap B)$ is the graph of $f$.
Here, the notation $\widetilde{\sigma}$ denotes the map defined in Definition \ref{def:x} and $\pi_d:M^n \rightarrow M^d$ denotes the projection onto the first $d$ coordinates.

The notation $\operatorname{Reg}(X)$ denotes the set of points at which $X$ is locally a good submanifold.
The notation $\operatorname{Sing}(X)$ denotes the singular locus defined by $X \setminus \operatorname{Reg}(X)$.
A semi-definable set is called a \textit{good submanifold} if it is a locally good submanifold at every point in it. 
 
Let $\pi:M^n \rightarrow M^{n-1}$ be the projection forgetting the last coordinate.
A semi-definable subset $X$ of $M^n$ is \textit{locally the graph of a continuous function} at $x \in X$ if there exists a bounded open box $U$ containing the point $x$ such that $\pi(X) \cap \pi(U)$ is a good submanifold and $X \cap U$ is the graph of a continuous function defined on $\pi(X) \cap \pi(U)$.
A semi-definable subset $X$ of $M^n$ is \textit{locally the graph of continuous functions everywhere} if it is locally the graph of a continuous function at every point in $X$.
\end{definition}

\begin{lemma}\label{lem:open_mfd}
Consider an o-minimal structure whose underlying space is $M$ and a good submanifold $X$ of $M^n$.
Let $U$ an open subset of $M^n$.
Then, $X \cap U$ is also a good submanifold.
\end{lemma}
\begin{proof}
Obvious.
\end{proof}

\begin{lemma}\label{lem:open_mfd2}
Consider an o-minimal structure whose underlying space is $M$.
Let $\pi:M^n \rightarrow M^{n-1}$ be the projection forgetting the last coordinate.
Take a semi-definable open subset $U$ of $M^{n-1}$.
If a semi-definable subset $X$ of $M^n$ is locally the graph of continuous functions everywhere, $X \cap (U \times M)$ is also locally the graph of continuous functions everywhere.
\end{lemma}
\begin{proof}
Obvious.
\end{proof}

\begin{lemma}\label{lem:reg_mfd}
Consider an o-minimal structure whose underlying space is $M$.
Let $X$ be a semi-definable set.
The set $\operatorname{Reg}(X)$ is an open semi-definable subset of $X$.
We also have $\dim(\operatorname{Sing}(X))<\dim (X)$.
\end{lemma}
\begin{proof}
It is obvious that $\operatorname{Reg}(X)$ is an open subset of $X$.
Take an arbitrary bounded open box $U$.
The set $\operatorname{Reg}(X) \cap U$ is obviously definable in the o-minimal structure because $X \cap U$ is definable.
It implies that $\operatorname{Reg}(X)$ is semi-definable.

Set $d=\dim(X)$.
Take an arbitrary bounded open box $B$.
We have only to show that $\dim (\operatorname{Sing}(X) \cap B) < d$ by Lemma \ref{lem:equiv_dim}.

Get a stratification of $\overline{B}$ partitioning $\partial B$, $X \cap B$ and $\operatorname{Sing}(X) \cap B$ by \cite[Chapter 4, Proposition 1.13]{vdD}.
Recall that a stratification of $\overline{B}$ is a partition of $\overline{B}$ into finitely many cells such that the frontier of a cell is a finite union of cells.
Let $C$ be an arbitrary cell contained in $X$ of dimension $d$.
The semi-definable set $X$ is a good submanifold of $M^n$ for any $x \in C$.
In fact, we have $C \cap U=X \cap U$ for any sufficiently small open box $U$ containing the point $x$.
Otherwise, there exists a cell $C'$ contained in $X$ such that $C' \cap U \neq \emptyset$ for any small open box $U$ containing the point $x$.
It means that $x \in \overline{C'}$.
We get $C \subseteq \partial C'$.
We have $\dim C'>d$ by \cite[Chapter 4, Theorem 1.8]{vdD}.
It means that $C' \cap X = \emptyset$.
Contradiction.
We have demonstrated $C$ is not contained in $\operatorname{Sing}(X) \cap B$.
It implies that that $\dim (\operatorname{Sing}(X) \cap B) < d$.
\end{proof}

\begin{lemma}\label{lem:cont_mfd}
Consider an o-minimal structure whose underlying space is $M$.
Let $\pi:M^n \rightarrow M^{n-1}$ be the coordinate projection forgetting the last coordinate.
Let $X$ be a semi-definable subset of $M^n$ such that $\pi(X)$ is semi-definable and a good submanifold, and the fiber $X \cap \pi^{-1}(x)$ is of dimension zero for any $x \in \pi(X)$.
We further assume that $\dim \pi(X)=\dim X$.
Let $S$ be the set of points at which $X$ is locally the graph of a continuous function.
Then, $S$ is semi-definable and we have $\dim(X \setminus S)<\dim(X)$.
\end{lemma}
\begin{proof}
It is obvious that $S$ is semi-definable.

Let $\mathcal R$ be the given o-minimal structure.
Set $d=\dim(X)=\dim \pi(X)$.
Take an arbitrary bounded open box $B$.
We have only to show that $\dim (T \cap B) < d$ by Lemma \ref{lem:equiv_dim}, where $T=X \setminus S$.

The intersections $X \cap B$ and $(\pi(X) \times M) \cap B$ are definable in $\mathcal R$.
We first apply the definable cell decomposition theorem for o-minimal structures \cite[Chapter 3, Theorem 2.11]{vdD}.
There exists a partition $\{C_1, \ldots, C_N\}$ of $B$ into cells partitioning $X \cap B$ and $(\pi(X) \times M) \cap B$.
Any cell $C$ contained in $X$ is the graph of a continuous function defined on $\pi(C)$.
Get a stratification of $\overline{\pi(B)}$ partitioning $\pi(C_1), \ldots, \pi(C_N)$ and $\pi(X) \cap \pi(B)$ by \cite[Chapter 4, Proposition 1.13]{vdD}.
Let $D_1, \ldots, D_L$ be the partition.
The family $\mathcal C=\{C_i \cap (D_j \times M)\;|\;1 \leq i \leq N, 1 \leq j \leq L, D_j \subset \pi(C_i)\}$ is a partition of $B$ into cells.

Take an arbitrary cell $C \in \mathcal C$ of dimension $d$ contained in $X$.
For any $x \in \pi(C)$, $\pi(X)$ is locally a good submanifold of $M^{m-1}$ at $x$ for the same reason as the proof of Lemma \ref{lem:reg_mfd}. 
Therefore, $X$ is locally the graph of a continuous function at every point in $C$ because $C$ is a cell.
We have shown that $\dim T<d$.
\end{proof}

\section{Geometry of almost o-minimal structures}\label{sec:almost}
We finally investigate almost o-minimal structures.
\subsection{Definably complete locally o-minimal structures}
We first introduce several lemmas on definably complete structures.
\begin{lemma}\label{lem:local0}
Let $\mathcal M=(M,<,\ldots)$ be a definably complete structure and $X$ be a definable subset of $M$.
Any open interval contained in $X$ is contained in a maximal open interval contained in $X$.
\end{lemma}
\begin{proof}
Let $I$ be an open interval contained in $X$.
Take a point $c$ with $c \in I$.
Set 
\begin{align*}
d &= \inf\{x \in M\;|\;(x<c) \wedge (\forall y, x<y<c \rightarrow y \in X)\} \in M \cup \{-\infty\}\text{ and }\\
e&= \sup\{x \in M\;|\;(x>c) \wedge (\forall y, c<y<x \rightarrow y \in X)\} \in M \cup \{\infty\}\text{.}
\end{align*}
They are well-defined because $\mathcal M$ is definably complete.
The open interval $]d,e[$ is obviously the maximal open interval containing the interval $I$ and contained in $X$.
\end{proof}

\begin{lemma}\label{lem:local1}
Consider a definably complete structure $\mathcal M=(M,<,\ldots)$.
The following are equivalent:
\begin{enumerate}
\item[(1)] The structure $\mathcal M$ is a locally o-minimal structure.
\item[(2)] Any definable set in $M$ either has a nonempty interior or it is closed and discrete. 
\end{enumerate}
\end{lemma}
\begin{proof}
\cite[Lemma 2.3]{Fuji4}.
\end{proof}

We then consider a definably complete locally o-minimal structures.
\begin{lemma}\label{lem:local2}
Let $\mathcal M=(M,<,\ldots)$ be a definably complete locally o-minimal structure and $X$ be a definable subset of $M$.
Any element $x$ in $M$ satisfies exactly one of the following conditions:
\begin{enumerate}
\item[(1)] The point $x$ is an element of an open interval contained in either $X$ or $M \setminus X$;
\item[(2)] The point $x$ is a discrete point of $X$;
\item[(3)] The point $x$ is an endpoint of a maximal open interval contained in $X$.
\end{enumerate}
The sets consisting of the discrete points of $X$ and consisting the endpoints of maximal open intervals contained in $X$ are definable, discrete and closed.
\end{lemma}
\begin{proof}
Let $x$ be an arbitrary element in $M$.
There exists an open interval $I$ containing the point $x$ such that $X \cap I$ is a finite union of points and open intervals.
Therefore, by Lemma \ref{lem:local0}, exactly one of (1) through (3) is obviously satisfied.
Consider the definable set $$Y =\{x \in M \;|\;\forall a\  \forall b, a<x<b \rightarrow (\exists y, \exists z \ a<y<b, a<z<b, y \in X, z \not\in X)\}\text{.}$$
The formula defining the set $Y$ is obviously the negation of the condition (1).
Consider the sets
\begin{align*}
D &=\{ x \in M\;|\; x \text{ is a discrete set in } X\} \text{ and }\\
E &=Y \setminus D\text{.}
\end{align*}
The set $D$ is the set consisting of the discrete points of $X$.
The set $E$ is the set consisting the endpoints of maximal open intervals contained in $X$.
They are both definable.
Since they do not contain an open interval, they are discrete and closed by Lemma \ref{lem:local1}. 
\end{proof}

Lemma \ref{lem:local2} provides tests for a definably complete locally o-minimal structure being o-minimal or almost o-minimal.
\begin{corollary}\label{cor:local1}
Let $\mathcal M=(M,<,\ldots)$ be a definably complete locally o-minimal structure.
The structure $\mathcal M$ is o-minimal if and only if any definable discrete subset of $M$ is a finite set.
\end{corollary}
\begin{proof}
We have only to show that, for any definable subset $X$ of $M$, the set of discrete points and the set consisting of the endpoints of maximal open intervals contained in $X$ are finite.
It is immediate from Lemma \ref{lem:local2}.
\end{proof}

\begin{corollary}\label{cor:local2}
Let $\mathcal M=(M,<,\ldots)$ be a definably complete locally o-minimal structure.
The structure $\mathcal M$ is almost o-minimal if and only if any bounded definable discrete subset of $M$ is a finite set.
\end{corollary}
\begin{proof}
We can prove it in the same manner as Corollary \ref{cor:local1}.
We omit the proof.
\end{proof}

\subsection{Basic properties of almost o-minimal structures}
We begin to study the basic properties of almost o-minimal structures.

\begin{lemma}\label{lem:almost1}
An almost o-minimal structure is definably complete.
\end{lemma}
\begin{proof}
Let $M$ be the universe of the considered structure and $X$ be a nonempty definable subset of $M$.
We demonstrate that $\sup(X)$ is well-defined and $\sup(X) \in M \cup \{\infty\}$.
Take a point $c \in X$ and consider the set $Y=\{x \in X\;|\;x \geq c\}$.
We may assume that $X$ is bounded from below by considering $Y$ instead of $X$.
When $X$ is unbounded, we have $\sup(X) =\infty$.
Otherwise, $X$ is bounded and it is a finite union of points and open intervals by the definition.
It obviously that $\sup(X)$ is well-defined and $\sup(X) \in M$.

We can prove that $\inf(X)$ is well-defined and $\inf(X) \in M \cup \{-\infty\}$, similarly.
\end{proof}

\begin{corollary}\label{lem:almost2}
An almost o-minimal structure is an o-minimal structure or has an unbounded infinite discrete definable set.
\end{corollary}
\begin{proof}
Immediate from Lemma \ref{lem:almost1}, Corollary \ref{cor:local1} and Corollary \ref{cor:local2}.
\end{proof}

\begin{proposition}\label{prop:almost1}
A locally o-minimal structure is o-minimal if and only if it is almost o-minimal and satisfies the type-completeness property defined in \cite{S}.
\end{proposition}
\begin{proof}
It immediately follows from \cite[Theorem 2.10, Corollary 2.11]{S}, Corollary \ref{cor:local1}, Corollary \ref{cor:local2} and Lemma \ref{lem:almost1}.
\end{proof}

An almost o-minimal structure is a uniformly locally o-minimal structure of the second kind by Corollary \ref{cor:second}.
An almost o-minimal expansion of an ordered field is o-minimal by \cite[Proposition 2.1]{Fuji}.
We do not expect that the multiplication is definable in a non-o-minimal almost o-minimal structure.
However, when we study bounded definable sets, even the multiplication definable only in the bounded regions is useful.
Therefore, we propose the following definition:
\begin{definition}
An expansion of dense linear order without endpoints $\mathcal M=(M,<,\ldots)$ \textit{has a bounded definable field structure} if there exist elements $0,1 \in M$ and binary maps $\oplus, \otimes: M \times M \rightarrow M$ such that the tuple $(M,<,0,1,\oplus,\otimes)$ is an ordered real closed field and, for any $a,b \in M$ with $a<b$ the restrictions $\oplus|_{]a,b[ \times ]a,b[}$ and $\otimes|_{]a,b[ \times ]a,b[}$ of the addition $\oplus$ and the multiplication $\otimes$ to $]a,b[ \times ]a,b[$ are definable in $\mathcal M$.
\end{definition}

We also need the following:
\begin{definition}
An ordered abelian group $(G, +, 0, <)$ is \textit{archimedean} if, for any positive $a,b \in G$, we have $na>b$ for some positive integer $n$. 
Here, $na$ denotes the summation of $n$ copies of $a$.
\end{definition}

We give examples of almost o-minimal structures.
\begin{proposition}\label{prop:almost2}
A locally o-minimal structure $\mathcal M =(M,<)$ is almost o-minimal if one of the following conditions is satisfied:
\begin{enumerate}
\item[(1)] All closed bounded intervals are compact;
\item[(2)] It is a uniformly locally o-minimal expansion of an ordered group of the second kind having bounded field structure;
\item[(3)] It is a definably complete locally o-minimal expansion of an archimedean ordered group, and the image of a nonempty definable discrete set under a coordinate projection is again discrete.
\end{enumerate}
\end{proposition}
\begin{proof}
(1) 
Obvious. We omit the proof.

(2) Let $X$ be a bounded definable set in $M$.
We have the addition and the multiplication $\oplus, \otimes: M \times M \rightarrow M$ whose restriction to the product of bounded open intervals definable in $\mathcal M$.
We may assume that $X$ is contained in the bounded closed interval $[-N,N]$. 
Consider the set $Y=\{(t,x) \in [0,1] \times M\;|\; \exists y \in X,\ x=t \otimes y\}$.
The set $Y$ is definable because we assume the bounded definable multiplication.
Since $\mathcal M$ is uniformly locally o-minimal of the second kind, there exists an open interval $I=]-L,L[$ containing the origin and a small positive $\varepsilon>0$ such that, for any $0<t<\varepsilon$, the intersection $I \cap Y_t$ of $I$ with the fiber $Y_t = \{t \otimes y \in M\;|\; y \in X\}$ of $Y$ at $t$ is a finite union of points and open intervals.
Take $t>0$ smaller than $\varepsilon$ and satisfying $L \otimes t <N$.
We have $Y_t \subseteq I$.
Since $Y_t=Y_t \cap I$ is a finite union of points and open intervals, $X$ is also a finite union of points and open intervals.

(3) Let $D$ be a nonempty bounded definable discrete set.
We have only to show that it is a finite set by Corollary \ref{cor:local2}.
It is also closed by \cite[Lemma 2.4]{Fuji4}.
We may assume that $D$ has at least two points without loss of generality.
Set $m=\sup(D)$.
We get $m<\infty$ because $D$ is bounded.
We have $m \in D$ because $D$ is closed.
Set $E=D \setminus \{m\}$.
Consider the successor function $\mysucc:E \rightarrow M$ given by
$$\mysucc(x)=\inf\{y \in D\;|\; y>x\}\text{.}$$
Since $D$ is closed, we have $\mysucc(x) \in D$.

Consider the definable function $\rho:E \rightarrow M$ defined by $\rho(x)=\mysucc(x)-x$.
We have $\rho(x)>0$.
The graph of the map $\rho$ is definable and discrete.
The image $\rho(E)$ is the projection image of the graph and it is also discrete by the assumption.
Since $\rho(E)$ is discrete, it is closed by \cite[Lemma 2.4]{Fuji4} again.
Set $d=\inf \rho(E)$.
We have $d \in \rho(E)$ because $\rho(E)$ is closed.
In particular, we have $d>0$.
Set $l=\sup(D)-\inf(D)$.
Since the structure is archimedean, there exists a positive integer $n$ with $nd>l$.
By the definition of $\rho$, the cardinality of the set $D$ is not greater than $n$.
\end{proof}

\begin{example}
We can easily construct a locally o-minimal structure having bounded definable field structure which is not an expansion of an ordered field.
Consider an arbitrary o-minimal expansion of the real field $\widetilde{\mathbb R}$.
The structure $[0,1)_{\text{def}}$ is the structure whose universe is $[0,1)$ defined in \cite[Definition 2]{KTTT}.
The simple product of $\mathbb Z$ and $[0,1)_{\text{def}}$ has bounded definable field structure but it is not an expansion of an ordered field.
The definition of a simple product is found in \cite[Definition 14]{KTTT}.
\end{example}

\begin{remark}
The latter condition in Proposition \ref{prop:almost2}(3) is \cite[Definition 1.1(a)]{Fuji4}. 
This condition is satisfied in a model of DCTC \cite[Corollary 4.3]{S} and in a definably complete uniformly locally o-minimal expansion of an ordered group of the second kind \cite[Proposition 2.12]{Fuji4}.
\end{remark}

Almost o-minimality does not preserve under elementary equivalence.
\begin{proposition}\label{prop:not_almost}
Let $\mathcal M=(M,<,\ldots)$ be an almost o-minimal structure which is not o-minimal.
An $\omega$-saturated elementary extension of $\mathcal M$ is not almost o-minimal.
\end{proposition}
\begin{proof}
There exists a definable discrete infinite subset $D$ of $M$ by Corollary \ref{lem:almost2}.
It is closed by Lemma \ref{lem:local1}.
Take an arbitrary element $c \in M$.
We assume that $D_{>c}=\{x \in D\;|\; x>c\}$ is infinite.
We can prove the proposition in the same manner when $D_{<c}=\{x \in D\;|\; x<c\}$ is infinite.

We consider the formula $\Phi_n(x)$ defining that the open interval $]c,x[$ has at least $n$ elements in $D$ for any non-negative integer $n$.
The family $\{\Phi_n(x)\}$ is finitely satisfiable in $\mathcal M$ because $D_{>c}$ is infinite.
Let $\mathcal N=(N,<,\ldots)$ be an $\omega$-saturated elementary extension of $\mathcal M$.
We can find $d \in N$ such that $\mathcal N \models \Phi_n(d)$ for all $n$.
In particular, the definable set $]c,d[ \cap D^{\mathcal N}$ is not a finite union of points and open intervals.
Here, the notation $D^{\mathcal N}$ be the subset of $N$ defined by the same formula as $D$.
It implies that $\mathcal N$ is not almost o-minimal.
\end{proof}

The definition of a structure having bounded definable field structure seems to be technical.
However, as indicated in the following proposition, sufficiently complex almost o-minimal structure has bounded definable field structure.
\begin{proposition}\label{prop:field_str}
Consider an almost o-minimal expansion of an ordered group.
Assume further that there exists a strictly monotone homeomorphism from an unbounded open interval to a bounded open interval the graph of whose restriction to arbitrary bounded open subintervals of the unbounded open interval is definable.
Then, any definable function is piecewise linear or the structure has bounded definable field structure.
\end{proposition}
\begin{proof}
An almost o-minimal expansion $\mathcal M$ of an ordered group is obviously an $\mathfrak X$-structure.
It is an $\mathfrak X$-expansion of an ordered divisible abelian group by Lemma \ref{lem:almost1} and \cite[Proposition 2.2]{M}.
Take an o-minimal expansion of an ordered group $\mathcal R$ given in Theorem \ref{thm:in_omin}. 
Any set definable in $\mathcal R$ is definable in $\mathcal M$ and any bounded set definable in $\mathcal M$ is definable in $\mathcal R$.

Consider the $\mathfrak X$-structure $\mathfrak X(\mathcal R)$ of semi-definable sets in $\mathcal R$.
There exists a strictly monotone homeomorphism between a bounded interval and an unbounded interval which is $\mathfrak X$-definable in $\mathfrak X(\mathcal R)$ by the assumption of the proposition.
By Theorem \ref{thm:xstr}, either any function $\mathfrak X$-definable in $\mathfrak X(\mathcal R)$ is piecewise linear or there exists binary maps $\oplus, \otimes:M^2 \rightarrow M$ $\mathfrak X$-definable in $\mathfrak X(\mathcal R)$ such that $(M,<,0,1,\oplus,\otimes)$ is an ordered real closed field.
In the former case, any function definable in $\mathcal M$ is piecewise linear because it is also $\mathfrak X$-definable in $\mathfrak X(\mathcal R)$

Consider the latter case.
The restriction of the addition $\oplus$ to the bounded open box $]a,b[ \times ]a,b[$ is definable in $\mathcal R$ and; therefore, it is definable in $\mathcal M$.
It is the same for the multiplication $\otimes$.
We have demonstrated that the structure $\mathcal M$ has bounded definable field structure.
\end{proof}

\subsection{Uniform local definable cell decomposition}
\subsubsection{Preliminary}
We begin to study uniform local definable cell decomposition for almost o-minimal structures.
The author developed the theory of uniformly locally o-minimal structures of the second kind and their dimension theory in \cite{Fuji,Fuji3,Fuji4}.
An almost o-minimal structure is a definably complete uniformly locally o-minimal structures of the second kind by Corollary \ref{cor:second} and Lemma \ref{lem:almost1}.
We use the following facts:

\begin{proposition}\label{prop:olddim}
Consider a DCULOAS structure $\mathcal M=(M, <, 0, +, \ldots)$.
The following assertions hold true:
\begin{enumerate}
\item[(1)] Let $f:X \rightarrow M^n$ be a definable map. 
We have $\dim(f(X)) \leq \dim X$.
\item[(2)] Let $f:X \rightarrow M^n$ be a definable map. 
The notation $\mathcal D(f)$ denotes the set of points at which the map $f$ is discontinuous. 
 The inequality $\dim(\mathcal D(f)) < \dim X$ holds true.
\item[(3)] (Addition Property)
Let $\varphi:X \rightarrow Y$ be a definable surjective map whose fibers are equi-dimensional; that is, the dimensions of the fibers $\varphi^{-1}(y)$ are constant.
We have $\dim X = \dim Y + \dim \varphi^{-1}(y)$ for all $y \in Y$.  
\end{enumerate}
\end{proposition}
\begin{proof}
(1) \cite[Theorem 1.1]{Fuji3};
(2) \cite[Corollary 1.2]{Fuji3};
(3) \cite[Theorem 3.14]{Fuji4}.
\end{proof}

We consider an almost o-minimal expansion of an ordered group $\mathcal M=(M, <$, $0, +, \ldots)$ in Section \ref{sec:multi} and Section \ref{sec:udcd}.
An $\mathcal M$-definable set is simply called definable in these sections. 
We introduce several notations.
The almost o-minimal structure $\mathcal M$ is simultaneously an $\mathfrak X$-expansion of an ordered divisible abelian group.
Applying Theorem \ref{thm:in_omin} to it, there exists an o-minimal expansion of an ordered group such that any set definable in this structure is definable in $\mathcal M$ and any bounded set definable in $\mathcal M$ is definable in the structure.
The notation $\myindr(\mathcal M)$ denote this o-minimal structure.
Sets semi-definable in $\myindr(\mathcal M)$ are simply called semi-definable in these sections.
Any definable set is semi-definable by the definitions of the o-minimal structure $\myindr(\mathcal M)$.

\subsubsection{Partition into multi-cells}\label{sec:multi}
\begin{definition}
Consider an expansion of a dense linear order $\mathcal M=(M,<,\cdots)$.
A subset $X$ of $M^{n+1}$ is called \textit{bounded in the last coordinate} if there exists a bounded open interval $I$ such that $X \subseteq M^n \times I$.
\end{definition}

\begin{lemma}\label{lem:ld2}
Consider an almost o-minimal structure $\mathcal M=(M,<,\ldots)$.
Let $X$ be a semi-definable subset of $M^{n+1}$ which is bounded in the last coordinate.
The image $\pi(X)$ is also semi-definable, where $\pi:M^{n+1} \rightarrow M^n$ is the projection forgetting the last coordinate.

In addition, if $X$ is semi-definably connected, then the image $\pi(X)$ is also semi-definably connected.
\end{lemma}
\begin{proof}
Obvious.
\end{proof}

We next define multi-cells.

\begin{definition}
Consider an almost o-minimal expansion of an ordered group $\mathcal M=(M,<,0,+,\ldots)$.
We define a \textit{multi-cell} $X$ in $M^n$ inductively.
\begin{itemize}
\item If $n=1$, either $X$ is a discrete definable set or all semi-definably connected components of the definable set $X$ are open intervals. 
\item When $n>1$, let $\pi:M^n \rightarrow M^{n-1}$ be the projection forgetting the last coordinate.
The projection image $\pi(X)$ is a multi-cell and, for any semi-definably connected component $Y$ of $X$, $\pi(Y)$ is a semi-definably connected component of $\pi(X)$ and $Y$ is one of the following forms:
\begin{align*}
Y&=\pi(Y) \times M  \text{,}\\
Y&=\{(x,y) \in \pi(Y) \times M \;|\; y=f(x)\} \text{,}\\
Y &= \{(x,y) \in \pi(Y) \times M \;|\; y>f(x)\} \text{,}\\
Y &= \{(x,y) \in \pi(Y) \times M \;|\; y<g(x)\} \text{ and }\\
Y &= \{(x,y) \in \pi(Y) \times M \;|\; f(x)<y<g(x)\}
\end{align*}
for some continuous functions $f$ and $g$ defined on $\pi(Y)$ with $f<g$.
\end{itemize} 
\end{definition}

A definable set is partitioned into finitely many multi-cells.
Its proof is long.
We divide the proof into several lemmas.

\begin{lemma}\label{lem:multi-cell-pre}
Consider an almost o-minimal expansion of an ordered group $\mathcal M=(M,<,0,+,\ldots)$ which is not o-minimal.
Let $X$ be a definable subset of $M^n$ with $n>1$ and $\pi:M^n \rightarrow M^{n-1}$ be the projection forgetting the last coordinate.
Assume that, for any $x \in M^{n-1}$, the fiber $X \cap \pi^{-1}(x)$ is at most of dimension zero.
Then, there exists a definable closed subset $Z$ of $M^{n-1}$ satisfying the following conditions:
 \begin{enumerate}
\item[(a)]  $\dim(Z) < \dim(\pi(X))$;
\item[(b)] The definable set $X \setminus \pi^{-1}(Z)$ is closed in $M^n \setminus \pi^{-1}(Z)$;
\item[(c)] The definable set $X \setminus \pi^{-1}(Z)$ is locally the graph of continuous functions everywhere;
\item[(d)] Any semi-definably connected component $C$ of $X \setminus \pi^{-1}(Z)$ is bounded in the last coordinate.
\end{enumerate}
\end{lemma}
\begin{proof}
Set $d=\dim(X)$.
We have $\dim(\pi(X))=d$ by Proposition \ref{prop:olddim}(3).
Let $Z_1=\pi(\partial X)$.
We have $\dim Z_1 \leq \dim \partial X < \dim X=d$ by Corollary \ref{cor:dim} and Proposition \ref{prop:olddim}(1).
Set $X_1 =X \setminus \pi^{-1}(Z_1)$.
We have $\dim X_1=d$ by Proposition \ref{prop:dim}(b).
The definable set $X_1$ is closed in $M^n \setminus \pi^{-1}(Z_1)$.
 
We now set $Z_2 = {\operatorname{Sing}(\pi(X_1))}$ and $X_2 = X_1 \setminus \pi^{-1}(Z_2)$.
They are obviously definable in $\mathcal M$.
The definable set $\pi(X) \setminus (Z_1 \cup Z_2)=\pi(X_1) \setminus Z_2 = \pi(X_2)$ is a good submanifold by Lemma \ref{lem:open_mfd} and  Lemma \ref{lem:reg_mfd}.
We get $\dim Z_2<d$ by Lemma \ref{lem:reg_mfd} and Corollary \ref{cor:dim}.
We also have $\dim X_2=d$ for the same reason as above.

We want to apply Lemma \ref{lem:cont_mfd} to $X_2$.
We have $\dim X_2=\dim \pi(X_2)$ by Proposition \ref{prop:olddim}(3).
The assumption in Lemma \ref{lem:cont_mfd} is satisfied.
Let $S$ be the set of points at which $X_2$ is locally the graph of a continuous function.
It is obviously definable in $\mathcal M$.
We have $\dim(X_2 \setminus S)<d$ by Lemma \ref{lem:cont_mfd}.
Set $Z_3=\overline{\pi(X_2 \setminus S)}$.
We have $\dim(Z_3)<d$ by Corollary \ref{cor:dim} and Proposition \ref{prop:olddim}(1).
The definable set $X_3=X_2 \setminus \pi^{-1}(Z_3)$ is locally the graph of a continuous function at any point in $X_3$.
We have $\dim X_3=d$.

There exists an unbounded discrete definable subset $D$ of $M$ by Corollary \ref{lem:almost2}.
We may assume that $\inf(D)=-\infty$ and $\sup(D)=\infty$ by considering $D \cup (-D)$ in place of $D$ because the group operation is definable.
Let $V_r$ be the boundary of $X_3 \cap (\pi(X_3) \times \{r\})$ in $\pi(X_3) \times \{r\}$ for any $r \in D$.
Set $W=\bigcup_{r \in D} V_r$ and $Z_4=\overline{\pi(W)}$.
The set $W$ is definable.
In fact, $W$ is given by $\{(x,r) \in M^{n-1} \times M\;|\; r \in D, \ \text{the point }x \text{ is contained in the boundary of } X_3 \cap \pi^{-1}(r) \text{ in }\pi(X_3) \times \{r\}\}$ and it is definable.

We demonstrate that $\dim W<d$.
Take an arbitrary bounded open box $B$ in $M^n$.
We have only to demonstrate that $\dim(B \cap W)<d$ by Lemma \ref{lem:equiv_dim}.
Note that $B \cap W$ is definable in $\myindr(\mathcal M)$.
Take an arbitrary cell $C$ contained in $B \cap W$.
We have to show that $\dim C<d$ by \cite[Chapter 3, (1.1)]{vdD}.
There exists $r \in D$ with $C \subseteq V_r$.
We have $\dim (C)< \dim (X_3 \cap (\pi(X_3) \times \{r\})) \leq \dim(X_3)$ by Corollary \ref{cor:dim}.
We get the inequality $\dim W<d$, and consequently, we obtain $\dim Z_4<d$ by Corollary \ref{cor:dim} and Proposition \ref{prop:olddim}(1).

Set $Z=Z_1 \cup Z_2 \cup Z_3 \cup Z_4$.
We are now ready to demonstrate that the conditions (a) through (d) are satisfied.
The condition (a) is now immediate by Proposition \ref{prop:dim}(b).
The condition (b) is satisfied because $\partial X$ is contained in $\pi^{-1}(Z)$.
The condition (c) follows from the definition of $Z_3$, Lemma \ref{lem:open_mfd} and Lemma \ref{lem:open_mfd2}.

The remaining task is to show that the condition (d) is satisfied.
Set $X_{\text{flat}}=X \cap \left((\pi(X) \setminus Z) \times D\right)$ and $X_{\text{flat},r}=X \cap ((\pi(X) \setminus Z) \times \{r\}) $ for all $r \in D$ for simplicity of notations.
By the condition (c) and the definition of $Z_4$, we obtain the following assertion.
\medskip

$(*)$: For any $x \in X_{\text{flat}}$, there exists an open box $U$ containing the point $x$ such that $X \cap U= X_{\text{flat},r} \cap U$ for some $r \in D$, $\pi(X) \cap \pi(U)$ is a good submanifold of $M^n$ and $X \cap U$ is the graph of a constant function defined on $\pi(X) \cap \pi(U)$.
\medskip

In fact, by the definition of $Z_4$, we can take an open box $U$ containing the point $x$ such that $X_{\text{flat}} \cap U= X_{\text{flat},r} \cap U$ for some $r \in D$ and $U \cap \pi^{-1}(Z_4) = \emptyset$.
Shrinking $U$ if necessary, we may assume that $\pi(X) \cap \pi(U)$ is a good submanifold of $M^n$ and $X \cap U$ is the graph of a continuous function on $\pi(X) \cap \pi(U)$ by the condition (c).
If $X \cap U$ is not the graph of a constant function, it intersects with $V_r$, and it means that $U \cap \pi^{-1}(Z_4) \neq \emptyset$.
Contradiction.
\medskip

Fix an arbitrary semi-definably connected component $C$ of $X \setminus \pi^{-1}(Z)$.
We consider two cases, separately.

We first consider the case in which $X_{\text{flat}} \cap C \neq \emptyset$.
We demonstrate that $C \subseteq X_{\text{flat},r}$ for some $r \in D$.
Take a point $x\in X_{\text{flat}} \cap C $.
The point $x$ is contained in $X_{\text{flat},r}$ for some $r \in D$.
We lead to a contradiction assuming that $y \not\in X_{\text{flat},r}$ for some $y \in C$.
There exists an $\myindr(\mathcal M)$-definable continuous curve $\gamma:[c_1,c_2] \rightarrow C$ such that $\gamma(c_1)=x$ and $\gamma(c_2)=y$ by Theorem \ref{thm:connected}.
The set $\gamma^{-1}(X_{\text{flat},r})=\gamma^{-1}(M^{n-1} \times \{r\})$ is a finite union of points and open intervals because it is definable in the o-minimal structure $\myindr(\mathcal M)$.
It is closed and does not coincide with the closed interval $[c_1,c_2]$ by the assumption.
Therefore, there exists a point $d \in [c_1,c_2]$ such that $\gamma(d) \in X_{\text{flat},r}$ and, for any small $\varepsilon>0$, we can take points $s_1,s_2 \in [c_1,c_2]$ satisfying that $|s_i-d|<\varepsilon$ for $i=1,2$, $\gamma(s_1) \in X_{\text{flat},r}$ and $\gamma(s_2) \not\in X_{\text{flat},r}$.
It implies that $X \cap U \neq X_{\text{flat},r} \cap U$ for any open box $U$ containing the point $\gamma(d)$.
It contradicts the assertion (*).
We have demonstrated that $C$ is contained in $X_{\text{flat},r}$.
Now, the condition (d) is clearly satisfied.

We next consider the case in which $X_{\text{flat}} \cap C = \emptyset$.
We demonstrate that $C \subseteq M^{n-1} \times ]r_1,r_2[$ for some $r_1,r_2 \in M$.
Let $\pi_2:M^n \rightarrow M$ be the projection onto the last coordinate.
We demonstrate that there exists $r_1 \in M$ such that $\pi_2(C) > r_1$.
Assume the contrary.
Take a point $x_1 \in C$, then there exists an $R \in D$ with $\pi_2(x_1) > R$ because $\inf(D)=-\infty$.
We can get $x_2 \in C$ with $\pi_2(x_2)<R$ by the assumption.
Then $C = \{x \in C\;|\; \pi_2(x)>R\} \cup \{x \in C\;|\; \pi_2(x)<R\}$ is a partition into two non-empty open and closed subsets.
It contradicts the assumption that $C$ is semi-definably connected.
We have demonstrated the existence of $r_1$.
We can take $r_2 \in M$ with $\pi(C)<r_2$ in the same manner.
This concludes the assertion (d).
\end{proof}

The following lemma is the major induction step of the proof of Theorem \ref{thm:multi-cell}.

\begin{lemma}\label{lem:multi-cell-pre2}
Consider an almost o-minimal expansion of an ordered group $\mathcal M=(M,<,0,+,\ldots)$ which is not o-minimal.
Let $X$ be a definable subset of $M^n$ and $\pi:M^n \rightarrow M^{n-1}$ be the projection forgetting the last coordinate.
Assume that, for any $x \in M^{n-1}$, the fiber $X \cap \pi^{-1}(x)$ is at most of dimension zero.
Assume further that any definable subset of $M^{n-1}$ is partitioned into finitely many multi-cells.

Then, the definable set $X$ is also partitioned into finitely many multi-cells.
Furthermore, the projection images of two distinct multi-cells are disjoint.
\end{lemma}
\begin{proof}
We prove the lemma by induction on $\dim(X)$.
When $\dim(X)=0$, $X$ is a discrete closed definable set and its projection images are also discrete and closed by Lemma \ref{lem:dim0} and Proposition \ref{prop:olddim}(1).
Therefore, X itself is a multi-cell.

Next we consider the case in which $\dim(X)>0$.
We can find a definable closed subset $Z$ of $M^{n-1}$ satisfying the following conditions by Lemma \ref{lem:multi-cell-pre}.
 \begin{enumerate}
\item[(a)]  $\dim(Z) < \dim(\pi(X))$;
\item[(b)] The definable set $X \setminus \pi^{-1}(Z)$ is closed in $M^n \setminus \pi^{-1}(Z)$;
\item[(c)] The definable set $X \setminus \pi^{-1}(Z)$ is locally the graph of continuous functions everywhere;
\item[(d)] Any semi-definably connected component $C$ of $X \setminus \pi^{-1}(Z)$ is bounded in the last coordinate.
\end{enumerate}

The lemma holds true for $X \cap \pi^{-1}(Z)$ by the induction hypothesis because $\dim(X \cap \pi^{-1}(Z)) \leq \dim(Z) < \dim(\pi(X))=\dim(X)$ by Proposition \ref{prop:olddim}(3).
Replacing $X$ with $X \setminus \pi^{-1}(Z)$, we may further assume the followings:
\begin{itemize}
\item $X$ is closed in $\pi^{-1}(\pi(X))$;
\item $X$ is locally the graph of continuous functions everywhere;
\item any semi-definably connected component of $X$ is bounded in the last coordinate.
\end{itemize}
We can partition $\pi(X)$ into finitely many multi-cells by the assumption.
Hence, we may assume that $\pi(X)$ is a multi-cell.
We demonstrate that $X$ is a multi-cell in this case.
Let $C$ be a semi-definably connected component of $X$.
We have only to show the following assertions:
\begin{itemize}
\item $\pi(C)$ is a semi-definably connected component of $\pi(X)$.
\item $C$ is the graph of a continuous function defined on $\pi(C)$.
\end{itemize}

We first demonstrate that $\pi(C)$ is a semi-definably connected component of $\pi(X)$.
The image $\pi(C)$ is semi-definable and semi-definably connected by Lemma \ref{lem:ld2} because $C$ is bounded in the last coordinate.
Therefore, we have only to show that $\pi(C)$ is open and closed in $\pi(X)$.

We first show that $\pi(C)$ is open.
Take an arbitrary point $x \in C$. 
There exists an open box $U$ containing the point $x$ such that $X \cap U$ is the graph of a continuous function on $\pi(X) \cap \pi(U)$.
Shrinking $U$ if necessary, we may assume that $U$ is bounded and $X \cap U$ is definably connected as a set definable in $\myindr(\mathcal M)$.
On the other hand, $C \cap U$ is a union of definably connected components of $X \cap U$ because $C$ is semi-definably connected.
It implies that $C \cap U=X \cap U$.
The intersection $C \cap U$ is the graph of a continuous function on $\pi(X) \cap \pi(U)$.
In particular, we have $\pi(X) \cap \pi(U)=\pi(C\cap U) \subseteq \pi(C) \cap \pi(U) \subseteq \pi(X) \cap \pi(U)$.
We get $\pi(C) \cap \pi(U) = \pi(X) \cap \pi(U)$.
We have demonstrated that $\pi(C)$ is open in $\pi(X)$.

We next demonstrate that $\pi(C)$ is closed in $\pi(X)$. 
Assume for contradiction that we can take a point $x$ in the frontier of $\pi(C)$ in $\pi(X)$.
There exists a continuous curve $\gamma:]0,\varepsilon[ \rightarrow \pi(C)$ definable in $\myindr(\mathcal M)$ with $\displaystyle\lim_{t \to 0}\gamma(t) = x$ by Corollary \ref{cor:curve_selection}.
Define $f_u:]0,\varepsilon[ \rightarrow M$ by $f_u(t)=\sup\{y \in M\;|\; (\gamma(t),y) \in C\}$.
The definable set $\{(t,y) \in ]0,\varepsilon[ \times M\;|\; (\gamma(t),y) \in C\}$ is definable in $\myindr(\mathcal M)$ because $C$ is bounded in the last coordinate.
Therefore, the function $f_u$ is definable in $\myindr(\mathcal M)$.
We may assume that $f_u$ is continuous and monotone by the monotonicity theorem for o-minimal structures \cite[Chapter 3, Theorem 1.2]{vdD} by taking a sufficiently small $\varepsilon>0$ if necessary.
The limit $y=\displaystyle\lim_{t \to 0} f_u(t)$ exists because the function $f_u$ definable in $\myindr(\mathcal M)$ is bounded and monotone.
We have $(x,y) \in X$ because $X$ is closed in $\pi^{-1}(\pi(X))$.
We get $(x,y) \in C$ because $C$ is closed in $X$ by Theorem \ref{thm:connected}.
It means $x \in \pi(C)$.
Contradiction to the assumption that $x$ is a point in the frontier of $\pi(C)$ in $\pi(X)$.

We next demonstrate that $C$ is the graph of a continuous function defined on $\pi(C)$. 
We have only to show that the restriction of $\pi$ to $C$ is injective by the condition (b).
Set 
\begin{equation*}
T=\{x \in \pi(C)\;|\; |\pi^{-1}(x) \cap C| > 1\}\text{.}
\end{equation*}
We have only to demonstrate that $T$ is an empty set.
We first show that $T$ is semi-definable. 
Consider the set $S=\{(x,y_1,y_2) \in M^{n-1} \times M \times M\;|\; (x,y_1) \in C \text{, } (x,y_2) \in C \text{ and } y_1 < y_2\}$. 
The semi-definable set $S$ is bounded in the last coordinate, and the image $S'$ of $S$ under the projection forgetting the last coordinate is semi-definable by Lemma \ref{lem:ld2}.
It is obvious that $S'$ is also bounded in the last coordinate and $T=\pi(S')$.
The set $T$ is semi-definable using Lemma \ref{lem:ld2} again.

The set $T$ is open in $\pi(C)$.
In fact, take an arbitrary point $x \in T$.
There exist $y_1<y_2 \in M$ with $(x,y_1),(x,y_2) \in C$.
By the condition (b), there exists an open box $B$ with $x \in B \cap \pi(C)$ such that $X \cap \pi^{-1}(B)$ contains the graphs of two continuous functions whose values at $x$ are $y_1$ and $y_2$, respectively.
Therefore, $B \cap \pi(C)$ is contained in $T$, and $T$ is open in $\pi(C)$.

We next show that $T$ is closed in $\pi(C)$.
Assume the contrary.
Take a point $x \in \pi(C) \cap \partial T$.
We can take the unique $y \in M$ with $(x,y) \in C$ because $x \not\in T$.
There exists a continuous curve $\gamma:]0,\varepsilon[ \rightarrow \pi(C) \cap T$ definable in $\myindr(\mathcal M)$ such that $\displaystyle\lim_{t \to 0}\gamma(t)=x$ by Corollary \ref{cor:curve_selection}.
We define the maps $\eta_u,\eta_l:]0,\varepsilon[ \rightarrow M$ by
\begin{align*}
&\eta_u(t)=\sup\{u \in M\;|\; (\gamma(t),u) \in C\} \text{ and }\\
&\eta_l(t)=\inf\{u \in M\;|\; (\gamma(t),u) \in C\} \text{.}
\end{align*}
They are well-defined because $C$ is bounded in the last coordinate.
Take a sufficiently small $\varepsilon>0$.
The two functions $\eta_u$ and $\eta_l$ are definable in $\myindr(\mathcal M)$ and continuous and they have the limits $y_u = \displaystyle\lim_{t \to 0} \eta_u(t) \in M$ and $y_l = \displaystyle\lim_{t \to 0} \eta_l(t) \in M$ for the same reason as above.
We have $\eta_u(t) \not=\eta_l(t)$ because $\gamma(t) \in T$.
We have $(x,y_u) \in C$ and $(x, y_l) \in C$ because $C$ is closed in $\pi^{-1}(\pi(C))$.
We therefore get $y=y_l=y_u$ because $x \not\in T$.
The condition (b) fails at $(x,y)$ because $\eta_u(t) \not=\eta_l(t)$.
Contradiction.
We have shown that $T$ is closed in $\pi(C)$.

Since $\pi(C)$ is semi-definably connected and $T$ is open and closed in $\pi(C)$, we have either $T=\pi(C)$ or $T=\emptyset$.
We have only to lead to a contradiction assuming that $T=\pi(C)$.
Define the function $f_u:\pi(C) \rightarrow M$ by $f_u(x)=\sup\{t\;|\; (x,t) \in C\}$.
We can easily show that its graph is a semi-definable set because $C$ is bounded in the last coordinate.
It is a continuous function.
In fact, let $\mathcal D$ be the set of all the points at which $f_u$ is discontinuous.
Take a point $x \in \mathcal D$.
Let $V$ be the intersection of $\pi^{-1}(x)$ with the closure of the graph of $f_u|_{\pi(C) \setminus \{x\}}$, where $f_u|_{\pi(C) \setminus \{x\}}$ denote the restriction of $f_u$ to $\pi(C) \setminus \{x\}$.
The closure of the graph of $f_u|_{\pi(C) \setminus \{x\}}$ is semi-definable by Lemma \ref{lem:frontier}.
The set $V$ is semi-definable and bounded.
Consequently, $V$ is definable in $\myindr(\mathcal M)$ by the definition of semi-definable sets.
There exists a point $(x,y) \in V$ with  $y \not= f_u(x)$ by the assumption.
Note that $(x,y) \in C$ because $C$ is closed in $\pi^{-1}(\pi(C))$.
Since $X$ is locally the graph of continuous functions everywhere, the set $C$ is locally the graphs of continuous functions $g$ and $h$ defined on a neighborhood of $x$ in $\pi(X)$ at $(x,y)$ and $(x,f_u(x))$, respectively. 
Take a sufficiently small $\varepsilon >0$. 
Since $g$ and $h$ are continuous and $g(x)<h(x)$, we have $g(x')+\varepsilon<h(x')$ if $x'$ is sufficiently close to $x$.
We also get $h(x') \leq f_u(x')$ by the definition of the function $f_u$. 
We then have $g(x')+\varepsilon < f_u(x')$ for any $x'$ sufficiently close to $x$ and 
we obtain $(x,y)=(x,g(x)) \not\in V_x$.
Contradiction.
We have demonstrated that the function $f_u$ is continuous.

Consider the graph $\{(x,y) \in C\;|\; y = f_u(x)\}$.
It is easy to prove that the graph is an open and closed proper subset of $C$ using the fact that $X$ is locally the graph of continuous functions everywhere.
Contradiction to the assumption that $C$ is semi-definably connected.
\end{proof}

The following theorem is the main theorem in this subsection.
\begin{theorem}\label{thm:multi-cell}
Consider an almost o-minimal expansion of an ordered group.
A definable set is partitioned into finitely many multi-cells.
\end{theorem}
\begin{proof}
Consider the case in which the structure in consideration is o-minimal.
A definable set is partitioned into finitely many cells by \cite[Chapter 3, Theorem 2.11]{vdD}.
It is also a partition into finitely many multi-cells because a cell is simultaneously a multi-cell.

We next consider the case in which the structure is not o-minimal.
Let $\mathcal M=(M,<,0,+,\ldots)$ be an almost o-minimal expansion of an ordered group.
Let $X$ be a definable subset of $M^n$.
We demonstrate that the set $X$ is partitioned into finitely many multi-cells.
We prove it by induction on $n$.
Consider the case in which $n=1$.
The theorem is clear when $X=\emptyset$ or $X=M$.
We consider the other cases.
Let $X_1$ be the union of all the maximal open intervals contained in $X$, which is definable.
In fact, the set $X_1$ is described as follows:
\begin{align*}
X_1 &=\{x \in X\;|\; \exists \varepsilon >0,\ \forall y \in M,\ |x-y| < \varepsilon \rightarrow y \in X\} \text{.}
\end{align*}
The set $X_2 = X \setminus X_1$ is the set of the isolated points and the endpoints of the maximal open intervals in $X$ and it is a discrete closed definable set by Lemma \ref{lem:almost1} and Lemma \ref{lem:local2}.
The decomposition $X=X_1 \cup X_2$ is a partition into multi-cells. 

We next consider the case in which $n>1$.
Let $\pi:M^n \rightarrow M^{n-1}$ be the projection forgetting the last coordinate.
Consider the sets 
\begin{align*}
&X_{\text{oi}}=\{(x,y) \in M^{n-1} \times M \;|\; \exists \varepsilon >0,\  \forall y' \in M,\  |y'-y|<\varepsilon \rightarrow (x,y') \in X\} \text{,}\\
&X_{\forall}=\{(x,y) \in M^{n-1} \times M \;|\; \forall y' \in M,\  (x,y') \in X\} \text{,}\\
&X'_{\infty}=\{(x,y) \in M^{n-1} \times M \;|\;  \forall y' \in M,\ y'>y \rightarrow (x,y') \in X\} \text{ and }\\
&X'_{-\infty}=\{(x,y) \in M^{n-1} \times M \;|\; \forall y' \in M,\ y'<y \rightarrow (x,y') \in X\} \text{.}
\end{align*}
The subscript $\text{oi}$ of $X_{\text{oi}}$ is an acronym of open intervals.
Set 
\begin{align*}
&X_{\text{boi}}=X_{\text{oi}} \setminus (X_{\forall} \cup X'_{\infty} \cup X'_{-\infty})\text{,}\\ &X_{\infty} = X'_{\infty} \setminus X_{\forall}\text{,}\\
&X_{-\infty} = X'_{-\infty} \setminus X_{\forall}\text{ and }\\
&X_{\text{pt}} = X \setminus (X_{\text{boi}} \cup X_{\infty} \cup  X_{-\infty} \cup X_{\forall})\text{.}
\end{align*}
The subscripts $\text{boi}$ of $X_{\text{boi}}$ and $\text{pt}$ of $X_{\text{pt}}$ represent bounded open intervals and points, respectively.
The definable set $X$ is partitioned as follows:
\begin{equation*}
X = X_{\text{boi}} \cup X_{\infty} \cup  X_{-\infty} \cup X_{\forall} \cup X_{\text{pt}}\text{.}
\end{equation*}
By the definition and Lemma \ref{lem:local2}, semi-definably connected components of non-empty fibers of $X_{\text{boi}}$,  $X_{\infty}$,  $X_{-\infty}$, $X_{\forall}$ and $X_{\text{pt}}$ are a bounded open interval, an open interval unbounded above and bounded below, an open interval bounded above and unbounded below, M and a point, respectively.

We have only to show that the above five definable sets are partitioned into multi-cells.
The definable set $X_{\text{pt}}$ is partitioned into multi-cells by Lemma \ref{lem:multi-cell-pre2}.
As to $X_{\forall}$, there exists a partition into multi-cells $\pi(X_{\forall})=\bigcup_{i=1}^k Y_i$ by the induction hypothesis.
Set $X_{\forall,i}= Y_i \times M$, then the partition $X_{\forall}=\bigcup_{i=1}^k X_{\forall,i}$ is a partition into multi-cells.

Consider the set 
\begin{equation*}
Y_{\infty}=\{(x,y) \in \pi(X_{\infty}) \times M\;|\; (x,y) \not\in X_{\infty},\ \forall y',\ y'>y \rightarrow (x,y') \in X_{\infty}\}\text{.}
\end{equation*}
The definable sets $Y_{\infty}$ consists of the lower endpoints of fibers of $X_{\infty}$.
In particular, $Y_\infty$ satisfies the assumption of Lemma \ref{lem:multi-cell-pre2}.
Let $Y_\infty=\bigcup_{i=1}^k Y_{\infty,i}$ be a partition into multi-cells given by Lemma \ref{lem:multi-cell-pre2}.
Set $X_{\infty,i}=X_{\infty} \cap \pi^{-1}(\pi(Y_{\infty,i}))$.
We claim that each definable set $X_{\infty, i}$, $1 \leq i \leq k$, is a multi-cell.
In fact, it is clear that the projection image $\pi(X_{\infty,i})$ is a multi-cell because $\pi(X_{\infty,i})=\pi(Y_{\infty,i})$.
Since $Y_{\infty,i}$ is a multi-cell, it is the graph of a continuous function $f$ defined on $\pi(Y_{\infty,i})$.
It is obvious that $X_{\infty,i}=\{(x,y) \in \pi(X_{\infty,i}) \times M\;|\; y>f(x)\}$ by the definition.
Hence, the definable set $X_{\infty,i}$ is a multi-cell, and $X_{\infty}=\bigcup_{i=1}^k X_{\infty,i}$ is a partition into multi-cells.

We can show that the definable set $X_{-\infty}$ is partitioned into multi-cells in the same way.

The remaining task is to demonstrate that $X_{\text{boi}}$ is partitioned into multi-cells.
We may assume the followings:
\begin{enumerate}
\item[(i)] All the semi-definably connected components of non-empty fibers of $X$ are bounded open intervals; 
\item[(ii)] For any $x \in \pi(X)$, the closures of two distinct semi-definably connected components of $X \cap \pi^{-1}(x)$ have an empty intersection.
\end{enumerate}
In fact we can assume (i) by setting $X=X_{\text{boi}}$.
Let us prove why we can also assume (ii).
Consider the definable set 
\begin{align*}
Y_{\text{both}} &= \{(x,y_1,y_2) \in \pi(X) \times M^2\;|\; (x,y_1) \not\in X, (x,y_2) \not\in X, y_1<y_2,\\
&\qquad \forall c,\ y_1 < c < y_2 \rightarrow (x,c) \in X\}\text{.}
\end{align*}
Set 
\begin{align*}
X_{\text{upper}} &= \{(x,y) \in \pi(X) \times M\;|\; \exists y_1, y_2, \ (x,y_1,y_2) \in Y_{\text{both}},\  (y_1+y_2)/2 <y <y_2\}\text{,}\\
X_{\text{middle}} &= \{(x,y) \in \pi(X) \times M\;|\; \exists y_1, y_2, \ (x,y_1,y_2) \in Y_{\text{both}}, \  y=(y_1+y_2)/2\}\text{ and }\\
X_{\text{lower}} &= \{(x,y) \in \pi(X) \times M\;|\; \exists y_1, y_2, \ (x,y_1,y_2) \in Y_{\text{both}},\  y_1<y<(y_1+y_2)/2\}\text{.}
\end{align*}
The definable set $X_{\text{middle}}$ can be partitioned into finitely many multi-cells by Lemma \ref{lem:multi-cell-pre2}.
The closures of two distinct semi-definably connected components of $X_{\text{upper}} \cap \pi^{-1}(x)$ have empty intersections for all $x \in \pi(X)$.
The fiber $X_{\text{lower}} \cap \pi^{-1}(x)$ also enjoys the same property.
Therefore, we may assume that the definable set $X$ satisfies the assumption (ii) by setting $X=X_{\text{upper}}$ and $X=X_{\text{lower}}$.
\medskip

Consider the definable sets 
\begin{align*}
Y_{\text{upper}} &= \{(x,y) \in \pi(X) \times M\;|\; (x,y) \not\in X,\  \exists \varepsilon >0,\ \forall c,\ y-\varepsilon < c < y\\
&\qquad \rightarrow (x,c) \in X\}\text{ and }\\
Y_{\text{lower}} &= \{(x,y) \in \pi(X) \times M\;|\; (x,y) \not\in X,\  \exists \varepsilon >0,\ \forall c,\ y < c < y+\varepsilon\\
&\qquad \rightarrow (x,c) \in X\}\text{.}
\end{align*}
For any $x \in \pi(X)$, the fiber $Y_{\text{upper}} \cap \pi^{-1}(x)$ is the set of the upper endpoints of the maximal open intervals contained in $X \cap \pi^{-1}(x)$ by the assumption (i).
The fiber $Y_{\text{lower}} \cap \pi^{-1}(x)$ is the set of the lower endpoints of the maximal open intervals.
By Lemma \ref{lem:multi-cell-pre2}, both $Y_{\text{upper}}$ and  $Y_{\text{lower}}$ are partitioned into finitely many multi-cells.
Let $Y_{\text{upper}}=\bigcup_{i=1}^k Y_{\text{upper},i}$ and $Y_{\text{lower}}=\bigcup_{j=1}^l Y_{\text{lower},j}$ be these partitions, respectively.
We have $\pi(Y_{\text{upper},i_1}) \cap \pi(Y_{\text{upper},i_2})=\emptyset$ by Lemma \ref{lem:multi-cell-pre2} if $i_1 \not=i_2$.
We may further assume that, for all $1 \leq i \leq k$ and $1 \leq j \leq l$, we have either $\pi(Y_{\text{upper},i})=\pi(Y_{\text{lower},j})$ or  $\pi(Y_{\text{upper},i}) \cap \pi(Y_{\text{lower},j}) = \emptyset$.
In fact, for all $1 \leq i \leq k$ and $1 \leq j \leq l$, the definable set $\pi(Y_{\text{upper},i}) \cap \pi(Y_{\text{lower},j})$ is partitioned as a finite union of multi-cells by the induction hypothesis.
Let $\pi(Y_{\text{upper},i}) \cap \pi(Y_{\text{lower},j})=\bigcup_{m=1}^{p(i,j)} Z_{ijm}$ be this partitions.
Set $Y_{\text{upper},ijm}=Y_{\text{upper},i} \cap \pi^{-1}(Z_{ijm})$ and $Y_{\text{lower},ijm}=Y_{\text{lower},j} \cap \pi^{-1}(Z_{ijm})$.
They are obviously multi-cells satisfying the requirement.

Set $X_i = X \cap \pi^{-1}(\pi(Y_{\text{upper},i}))$.
We have a partition $X=\bigcup_{i=1}^k X_i$.
The remaining task is to show that each $X_i$ is a multi-cell.
Take an arbitrary semi-definably connected component $C$ of $X_i$ and an arbitrary point $\hat{z} \in C$.
Set $\hat{x}=\pi(\hat{z})$ and $\hat{z}=(\hat{x},\hat{y})$ for some $\hat{y} \in M$.
Since semi-definably connected components of the fiber $X \cap \pi^{-1}(\hat{x})$ are bounded open intervals by the assumption (i), there exist $y_u, y_l \in M$, $1 \leq i' \leq k$ and $1 \leq j' \leq l$ with $y_l < \hat{y} < y_u$, $(\hat{x},y_u) \in Y_{\text{upper},i'}$, $(\hat{x},y_l) \in Y_{\text{lower},j'}$ and $(\hat{x},y) \in X$ for all $y_l<y<y_u$.
We have $\pi(Y_{\text{upper},i'})=\pi(Y_{\text{lower},j'})$ by the assumption.
Let $Z$ be its semi-definably connected component of $\pi(X_i)$ containing the point $\hat{x}$.
There are two continuous function $f$ and $g$ defined on $Z$ such that $y_l=f(\hat{x})$, $y_u=g(\hat{x})$ and the graphs of $f$ and $g$ are semi-definably connected components of $Y_{\text{lower},j'}$ and $Y_{\text{upper},i'}$, respectively, because $Y_{\text{lower},j'}$ and $Y_{\text{upper},i'}$ are multi-cells.

We demonstrate that $f(x)<g(x)$ on $Z$ and $$C=\{(x,y) \in Z \times M\;|\; f(x)<y<g(x)\}\text{.}$$
We first show that the graph of $f$ does not intersect $Y_{\text{upper}}$. 
In particular, we have $f(x)<g(x)$ on $Z$ by Lemma \ref{lem:intermediate}.
Assume the contrary.
Let $x' \in Z$ and $y'=f(x')$ with $(x',y') \in Y_{\text{upper}}$. 
By the definition of $f$ and $Y_{\text{upper}}$, there exist $y_1,y_2 \in M$ with $y_1<y'<y_2$ such that $\{x\} \times ]y_1,y'[$ and $\{x\} \times ]y',y_2[$ are semi-definably connected components of the fiber $X \cap \pi^{-1}(x)$.
The intersection of their closures is not empty.
This contradicts (ii).

We next show that $C=\{(x,y) \in Z \times M\;|\; f(x)<y<g(x)\}$.
The set $C$ is contained in $\{(x,y) \in Z \times M\;|\; f(x)<y<g(x)\}$ because the latter set is closed and open in $X$ by the definition.
We demonstrate the opposite inclusion.
Assume the contrary.
Let $(x',y')$ be a point satisfying $x' \in Z$, $f(x')<y'<g(x')$ and $(x',y') \not\in C$.
By the assumption (i), there exists $\overline{y} \in M$ with $f(x')<\overline{y} \leq y'$ and 
$(x',\overline{y}) \in Y_{\text{upper}}$.
Since we have $\pi(Y_{\text{upper},i_1}) \cap  \pi(Y_{\text{upper},i_2})=\emptyset$ for all $i_1 \not= i_2$, we have $(x',\overline{y}) \in Y_{\text{upper},i'}$.
Since $Y_{\text{upper},i'}$ is a multi-cell, the semi-definably connected component of $Y_{\text{upper},i'}$ containing the point $(x',\overline{y})$ is the graph of some continuous function $g'$ defined on $Z$.
We have $f(x')<g'(x')<g(x')$.
The graph of $g'$ does not intersect the graph of $g$ because $Y_{\text{upper},i'}$ is a multi-cell.
The graph of $g'$ does not intersect the graph of $f$ because the graph of $f$ does not intersect $Y_{\text{upper}}$ as we demonstrated previously. 
We get $y_l=f(\hat{x})<g'(\hat{x})<g(\hat{x})=y_u$ by Lemma \ref{lem:intermediate}.
We obtain $(\hat{x},g'(\hat{x})) \not\in X$, which contradicts the fact that $(\hat{x},y) \in X$ for all $y_l<y<y_u$.
\end{proof}

\begin{remark}
The notion of special submanifolds defined in \cite{F,M2,T,Fuji4} is similar to that of multi-cells.

Consider an expansion of a densely linearly order without endpoints $\mathcal M=(M,<$,$\ldots)$.
Let $\pi:M^n \rightarrow M^d$ be a coordinate projection.
A definable subset is a \textit{$\pi$-special submanifold} or simply a \textit{special submanifold} if, $\pi(X)$ is a definable open set and, for every point $x \in \pi(X)$, there exists an open box $U$ in $M^d$ containing the point $x$ satisfying the following condition:
For any $y \in X \cap \pi^{-1}(x)$, there exist an open box $V$ in $M^n$ and a definable continuous map $\tau:U \rightarrow M^n$ such that $\pi(V)=U$, $\tau(U)=X \cap V$ and the composition $\pi \circ \tau$ is the identity map on $U$.

Let $\{X_i\}_{i=1}^m$ be a finite family of definable subsets of $M^n$.
A \textit{decomposition of $M^n$ into special submanifolds partitioning $\{X_i\}_{i=1}^m$} is a finite family of special submanifolds $\{C_i\}_{i=1}^N$ such that $\bigcup_{i=1}^NC_i =M^n$, $C_i \cap C_j=\emptyset$ when $i \not=j$ and either $C_i$ has an empty intersection with $X_j$ or is contained in $X_j$ for any $1 \leq i \leq m$ and $1 \leq j \leq N$.

For instance, a DCULOAS structure admits decomposition into special submanifolds \cite{Fuji4}.
A $d$-minimal expansion of an ordered field also admits decomposition into special submanifolds \cite{M2,T}.

A multi-cell is a special manifold, but the converse is false.
The projection image of a multi-cell under the projection forgetting the last coordinate  is again a multi-cell, but it is not true for a special manifold. 
We need a decomposition into multi-cells in order to prove Theorem \ref{thm:uniform}.
\end{remark}

\subsubsection{Uniform local definable cell decomposition}\label{sec:udcd}

In this subsection, we first show that an almost o-minimal expansion of an ordered group $\mathcal M=(M,<,0,+,\ldots)$ has a uniformity property.
We also prove the uniform local definable cell decomposition theorem introduced in Section \ref{sec:intro} using this uniformity property.

We need the following technical definition.

\begin{definition}
Consider an almost o-minimal expansion of an ordered group $\mathcal M=(M,<,0,+,\ldots)$.
Let $X \subseteq M^n$ be a multi-cell and $Y$ be a discrete definable subset of $X$.
Let $\pi_k:M^n \rightarrow M^k$ denote the projection onto the first $k$ coordinates for all $1 \leq k \leq n$.
Note that  $\pi_n$ is the identity map.
The definable set $Y$ is a \textit{representative set of semi-definably connected components of $X$} if the intersection of $\pi_k(Y)$ with any semi-definably connected component of $\pi_k(X)$ is a singleton for any $1 \leq k \leq n$. 
\end{definition}

\begin{lemma}\label{lem:onept}
Consider an almost o-minimal expansion of an ordered group $\mathcal M=(M,<,0,+,\ldots)$.
Let $X \subseteq M^{m+n}$ be a multi-cell and $\pi:M^{m+n} \rightarrow M^m$ be the projection onto the first $m$ coordinates.
There exists a definable subset $Y$ of $X$ such that $Y \cap \pi^{-1}(x)$ is a representative set of semi-definably connected components of $X \cap \pi^{-1}(x)$ for any $x \in \pi(X)$. 
\end{lemma}
\begin{proof}
We demonstrate the lemma by the induction on $n$.
We first consider the case in which $n=1$.
Consider the following definable sets:
\begin{align*}
&S_{\infty} = \{x \in \pi(X)\;|\; \forall y \in M,\ (x,y) \in X\}\text{,}\\
&S_u = \{x \in \pi(X)\;|\; \exists y \in M, \ \forall z, \ z > y \rightarrow (x,z) \in X\} \setminus S_\infty \text{ and }\\
&S_l = \{x \in \pi(X)\;|\; \exists y \in M, \ \forall z, \ z < y \rightarrow (x,z) \in X\} \setminus S_\infty\text{.}
\end{align*}
The definable functions $\rho_u:S_u \rightarrow M$ and $\rho_l:S_l \rightarrow M$ are given as follows:
\begin{align*}
\rho_u(x) &= \inf\{y \in M\;|\; \forall z, \ z > y \rightarrow (x,z) \in X\}\text{ and }\\
\rho_l(x) &= \sup\{y \in M\;|\; \forall z, \ z < y \rightarrow (x,z) \in X\}\text{.}
\end{align*}
It is well-defined by Lemma \ref{lem:almost1}.
We set
\begin{align*}
Y_c &= \{(x,y_1,y_2) \in \pi(X) \times M^2\;|\; (x,y_1) \not\in X, (x,y_2) \not\in X, y_1<y_2,\\
&\qquad \forall c,\ y_1 < c < y_2 \rightarrow (x,c) \in X\}\text{ and }\\
Y_p &= \{(x,y) \in X\;|\; \exists \varepsilon > 0,\  \forall c,\ 0<|y-c|<\varepsilon \rightarrow (x,c) \not\in X\}\text{.}
\end{align*}
We finally set
\begin{align*}
Y &= \{(x,\rho_u(x)+\varepsilon) \in M^{m+1}\;|\; x \in S_u\} \cup \{(x,\rho_l(x)-\varepsilon) \in M^{m+1}\;|\; x \in S_l\}\\
&\quad  \cup \{(x,y) \in M^{m+1}\;|\; \exists y_1, y_2, \ (x,y_1,y_2) \in Y_c,\ y = (y_1+y_2)/2\}\\
&\quad  \cup Y_p \cup (S_\infty \times \{0\})\text{,}
\end{align*}
where $\varepsilon$ is a fixed positive element in $M$.
The definable set $Y \cap \pi^{-1}(x)$ is obviously a representative set of semi-definably connected components of $X \cap \pi^{-1}(x)$ for any $x \in \pi(X)$ by the definition of multi-cells. 

We consider the case in which $n>1$.
The notations $\pi_1:M^{m+n} \rightarrow M^{m+n-1}$ and $\pi_2: M^{m+n-1} \rightarrow M^{m}$ denote the projections forgetting the last coordinate and onto the first $m$ coordinates, respectively.
The projection image $\pi_1(X)$ is a multi-cell by the definition of multi-cells.
There exists a definable subset $Y_1 \subseteq \pi_1(X)$ such that the definable set $Y_1 \cap \pi_2^{-1}(x)$ is a representative set of semi-definably connected components of $\pi_1(X) \cap \pi_2^{-1}(x)$ for any $x \in \pi(X)$ by applying the induction hypothesis to $\pi_1(X)$ and $\pi_2$.
Set $X'=X \cap \pi_1^{-1}(Y_1)$, and apply the lemma for $n=1$ to $X'$ and $\pi_1$.
We can find a representative set $Y$ of semi-definably connected components of $X'$.
It is easy to demonstrate that $Y$ is also a representative set of semi-definably connected components of $X$.
\end{proof}

\begin{theorem}[Uniformity theorem]\label{thm:uniform}
Consider an almost o-minimal expansion of an ordered group $\mathcal M=(M,<,0,+,\ldots)$.
For any definable subset $X$ of $M^{n+1}$ and a positive element $R \in M$, there exists a positive integer $K$ such that, for any $a \in M^n$, the definable set $X \cap (\{a\} \times ]-R,R[)$ has at most $K$ semi-definably connected components.
\end{theorem}
\begin{proof}
Consider the set $X^{<R}:= X \cap (M^n \times ]-R,R[)$.
Apply Theorem \ref{thm:multi-cell} to $X^{<R}$.
We have a partition into multi-cells $X^{<R} = \bigcup_{i=1}^k X_i$.
Let $\pi_1:M^{n+1} \rightarrow M^n$ and $\pi_2:M^{n+1} \rightarrow M$ be the projections onto first $n$ coordinates and onto the last coordinate, respectively.
We next apply Lemma \ref{lem:onept} to $X_i$ and $\pi_2$.
For any $1 \leq i \leq k$, we can take a definable discrete subset $Y_i$ of $X_i$ which is a representative set of semi-definably connected components of $X_i$.
Since $Y_i$ is discrete, we have $\dim (Y_i) \leq 0$ by Lemma \ref{lem:equiv_dim}.
Set $Z_i = \pi_2(Y_i)$ for all $1 \leq i \leq k$.
We get $\dim(Z_i) \leq 0$ by Proposition \ref{prop:olddim}(1).
It implies that the definable set $Z_i$ is discrete.
The definable set $Z_i$ is included in the bounded open interval $]-R,R[$ by the definition.
Hence, the definable set $Z_i$ is a finite set for any $1 \leq i \leq k$ because $\mathcal M$ is almost o-minimal.
Set $K=\sum_{i=1}^k |Z_i|$.

When $a \in \pi_1(X_i)$ for some $1 \leq i \leq k$, there exists a point $a'_i \in \pi_1(Y_i)$ contained in the semi-definably connected component of $\pi_1(X_i)$ containing the point $a$.
The definable set $X_i \cap \pi_1^{-1}(a)$ has the same number of semi-definably connected components as $X_i \cap \pi_1^{-1}(a'_i)$ which is equal to $|Y_i \cap \pi_1^{-1}(a'_i)|$ 
by the definitions of multi-cells and representative sets of their semi-definably connected components.
Let $\operatorname{NC}(S)$ denote the number of semi-definably connected components of a definable set $S$.
We therefore have
\begin{align*}
\operatorname{NC}(X \cap (\{a\} \times ]-R,R[)) &= \operatorname{NC}(X^{<R} \cap \pi_1^{-1}(a))
\leq \displaystyle\sum_{1 \leq i \leq k, a \in \pi_1(X_i)}\operatorname{NC}(X_i \cap \pi_1^{-1}(a))\\
&=\displaystyle\sum_{1 \leq i \leq k, a \in \pi_1(X_i)} \operatorname{NC}(X_i \cap \pi_1^{-1}(a'_i))\\
&=\displaystyle\sum_{1 \leq i \leq k, a \in \pi_1(X_i)} |Y_i \cap \pi_1^{-1}(a'_i)|\\
& \leq \displaystyle\sum_{i=1}^k |\pi_2(Y_i)| = \displaystyle\sum_{i=1}^k \left|Z_i\right|=K\text{.}
\end{align*}
We have finished the proof.
\end{proof}

We now begin to demonstrate Theorem \ref{thm:main}.

\begin{proof}[Proof of Theorem \ref{thm:main}]
We first show the assertion for $n=1$.
For any definable set $S \subseteq M^{m+1}$,  the notation $\operatorname{bd}_m(S)$ denotes the set $\{(x,y) \in M^m \times M\;|\; y \in \operatorname{bd}(S_x)\}$.
Set $I=]\!-\!R,R[$, then $S' \cap I$ is a finite union of points and open intervals for any definable subset $S'$ of $M$ by the definition of almost o-minimality.
Set $X= \bigcup_{\lambda \in \Lambda}\operatorname{bd}_m(A_\lambda \cap I)$.
The fibers $X_b$ are finite sets for all $b \in M^m$.
It is obvious that any definable cell decomposition of $I$ partitioning $X_b \cap I$ partitions $\{(A_\lambda)_b \cap I\}_{\lambda\in\Lambda}$ for any point $b \in M^m$.

There exists a positive integer $K$ such that $|X \cap(\{b\} \times  I)| \leq K$ for any point $b \in M^m$ by Theorem \ref{thm:uniform}.
Set $S_i=\{b \in M^m\;|\; |X_b \cap I|=i\}$ for all $0 \leq i \leq K$.
The family $\{S_i\}_{i=0}^K$ partitions the parameter space $M^m$.
Let $y_j(b)$ be the $j$-th largest point of $X_b \cap I$ for all $b \in S_i$ and $1 \leq j \leq i$.
Set $y_0(b)=-R$ and $y_{i+1}(b)=R$ for all $b \in S_i$.
Applying Proposition \ref{prop:olddim}(2) inductively, we can find a partition into definable sets 
\begin{equation*}
S_i = S_{i0} \cup \ldots \cup S_{im}
\end{equation*}
such that either $S_{ik}=\emptyset$ or $\dim(S_{ik})=k$, and $y_j$ is continuous on $S_{ik}$ for any $0 \leq j \leq i$ and $0 \leq k \leq m$.
We set
\begin{align*}
&C_{ijk} =\{(x,y_j(x)) \in S_{ik} \times M\} \ \ (1 \leq j \leq  i)\text{ and }\\
&D_{ijk} = \{(x,y) \in S_{ik} \times M\;|\; y_j(x) < y < y_{j+1}(x)\} \ \ (0 \leq j \leq i)
\end{align*}
for any $0 \leq i \leq K$ and $0 \leq k \leq m$. 
Consider the family of maps $\mathcal F = \{\sigma: \Lambda \rightarrow \{0,1\}\}$.
Set 
\begin{align*}
&T^0_{ijk\sigma} =\{x \in S_{ik}\;|\; C_{ijk} \cap (\{x\} \times M) \text{ is contained in } A_\lambda \text{ iff } \sigma(\lambda)=1\} \ (1 \leq j \leq  i) \text{ and }\\
&T^1_{ijk\sigma} =\{x \in S_{ik}\;|\; D_{ijk} \cap (\{x\} \times M) \text{ is contained in } A_\lambda \text{ iff } \sigma(\lambda)=1\} \ (0 \leq j \leq i)
\end{align*}
for any $0 \leq i \leq K$, $0 \leq k \leq m$ and $\sigma \in \mathcal F$.
We finally set $C_{ijk\sigma}=C_{ijk} \cap (T^0_{ijk\sigma} \times M)$ and $D_{ijk\sigma}=D_{ijk} \cap (T^1_{ijk\sigma} \times M)$.
The partition 
\begin{equation*}
M^m \times I = \bigcup_{i=0}^K \left(\bigcup_{k=1}^m \left(\bigcup_{\sigma \in \mathcal F}\left(\bigcup_{j=1}^i C_{ijk\sigma} \cup \bigcup_{j=0}^i D_{ijk\sigma}\right)\right)\right)
\end{equation*}
 is the desired partition.
Furthermore, the above definable functions $y_j$ can be chosen as continuous functions on $p(C_{ijk\sigma})$ and $p(D_{ijk\sigma})$, where $p:M^{m+1} \rightarrow M^m$ is the projection forgetting the last coordinate.
It is clear that the type of the cell $(X_i)_b$ is independent of the choice of $b$ with $(X_i)_b \not= \emptyset$.

We consider the case in which $n>1$.
Let $\pi:M^{m+n} \rightarrow M^{m+n-1}$ be the projection forgetting the last coordinate.
Set $I=]-R,R[$.
Applying the theorem for $n=1$ to the family $\{A_\lambda\}_{\lambda\in\Lambda}$, there exists a partition $M^{m+n-1} \times I = Y_1 \cup \ldots \cup Y_l$ such that  $I=(Y_1)_b \cup \ldots \cup (Y_l)_b$ is a definable cell decomposition $I$ for any $b \in M^{m+n-1}$ and either $Y_i \subseteq A_\lambda$ or $Y_i \cap A_\lambda=\emptyset$ for any $1 \leq i \leq l$ and $\lambda \in \Lambda$. 
We can further assume that $Y_i$ is one of the following forms:
\begin{align*}
Y_i &= \{(x,f(x)) \in \pi(Y_i) \times M\} \text{ and }\\
Y_i &= \{(x,y) \in \pi(Y_i) \times M\;|\; f(x)<y<g(x)\}\text{,}
\end{align*}
where $f$ and $g$ are definable continuous functions on $\pi(Y_i)$ with $f<g$.

Set $B'=]-R,R[^{n-1}$.
Apply the induction hypothesis to the family $\{\pi(Y_i)\}_{i=1}^l$.
There exists a partition $M^{m} \times B' = Z_1 \cup \ldots \cup Z_q$ such that  $B'=(Z_1)_b \cup \ldots \cup (Z_q)_b$ is a definable cell decomposition $B'$ for any $b \in M^{m}$, and either $\pi(Y_i) \cap Z_j = \emptyset$ or $Z_j \subseteq \pi(Y_i)$ and the type of the cell $(Z_j)_b$ is independent of the choice of $b$ with $(Z_j)_b \not= \emptyset$ for all $i$ and $j$.

Set $X_{ij}=Y_i \cap \pi^{-1}(Z_j)$ for all $1 \leq i \leq l$ and $1 \leq j \leq q$.
Let $\{X_i\}_{i=1}^k$ be the family of non-empty $X_{ij}$'s.
It is easy to demonstrate that the family $\{X_i\}_{i=1}^k$ satisfies the requirement of the theorem.
We omit the details. 
\end{proof}

\begin{corollary}
Consider an almost o-minimal expansion of an ordered group $\mathcal M=(M,<,0,+,\ldots)$.
For any definable subset $X$ of $M^n$ and a positive element $R \in M$, there exists a positive integer $K$ such that the definable set $X \cap (b+B)$ has at most $K$ semi-definably connected components for all $b \in M^n$.
Here, $B=]\!-\!R,R[^n$ and $b+B$ denotes the set given by $\{x \in M^n\;|\; x-b \in B\}$.
\end{corollary}
\begin{proof}
Consider the definable set $Y$ defined by 
\begin{equation*}
\{(y,x) \in M^n \times M^n\;|\; x-y \in X\} \text{.} 
\end{equation*}
Applying Theorem \ref{thm:main}, there exists a partition $M^{n} \times B = X_1 \cup \ldots \cup X_K$ such that  $B=(X_1)_b \cup \ldots \cup (X_K)_b$ is a definable cell decomposition $B$ partitioning the definable set $Y_b \cap B$ for any $b \in M^{n}$.
It means that the definable set $X \cap (b+B)$ is the union of at most $K$ cells.
The set $X \cap (b+B)$ has at most $K$ semi-definably connected components because cells are semi-definably connected.
We have finished the proof.
\end{proof}

\subsection{Structures elementarily equivalent to an almost o-minimal structure}

Structures elementarily equivalent to an almost o-minimal structure is not necessarily almost o-minimal as demonstrated in Proposition \ref{prop:not_almost}.
However, a weaker version of Theorem \ref{thm:main} holds true for such structures.

\begin{lemma}\label{lem:last}
Let $\mathcal M=(M,<,\ldots)$ be an expansion of a dense linear order.
Consider a definable set $C \subseteq M^n$ defined by a first-order formula with parameter $\overline{c}$.
There exists a first-order sentence with parameters $\overline{c}$ expressing the condition for $C$ being a definable cell of type $(j_1, \ldots, j_d)$.
\end{lemma}
\begin{proof}
We prove the lemma by the induction on $n$.
When $n=1$, the definable set $C$ is a cell if and only if $C$ is a point or an open interval.
This condition is clearly expressed by a first-order sentence.

We next consider the case in which $n>1$.
The notation $\pi:M^n \rightarrow M^{n-1}$ denotes the projection forgetting the last factor.
The condition for $\pi(C)$ being a cell is represented by a first-order sentence with parameters $\overline{c}$ by the induction hypothesis.
We only prove the lemma in the case in which the definable set $C$ is of the form
\begin{equation*}
C=\{(x,y) \in M^{n-1} \times M\;|\; f(x) < y < g(x)\}\text{,}
\end{equation*}
where $f$ and $g$ are definable continuous functions defined on $\pi(C)$.
We can demonstrate the lemma in the other cases in a similar way.
The above condition is equivalent to the following conditions:
\begin{itemize}
\item For any $x \in \pi(C)$, the fiber $C_x=\{y \in M\;|\; (x,y) \in C\}$ is a bounded interval.
\item Set $f(x)=\inf\{y \in M\;|\;(x,y) \in C\}$ and $g(x)=\sup\{y \in M\;|\;(x,y) \in C\}$ for any $x \in \pi(C)$, then $f$ and $g$ are continuous on $\pi(C)$.
\end{itemize}
The above conditions are obviously expressed by first-order sentences with parameters $\overline{c}$.
\end{proof}

\begin{theorem}\label{thm:udcd}
Consider a structure $\mathcal M = (M, <,0,+, \ldots)$ elementarily equivalent to an almost o-minimal expansion of an ordered group.
Let $\{A_\lambda\}_{\lambda\in\Lambda}$ be a finite family of definable subsets of $M^{m+n}$.
There exist an open box $B$ in $M^n$ containing the origin and a finite partition into definable sets 
\begin{equation*}
M^m \times B = X_1 \cup \ldots \cup X_k
\end{equation*}
such that $B=(X_1)_b \cup \ldots \cup (X_k)_b$ is a definable cell decomposition of $B$ for any $b \in M^m$ and either $X_i \cap A_\lambda = \emptyset$ or $X_i \subseteq A_\lambda$ for any $1 \leq i \leq k$ and $\lambda \in \Lambda$.
Here, the notation $S_b$ denotes the fiber of a definable subset $S$ of $M^{m+n}$ at $b \in M^m$.
\end{theorem}
\begin{proof}
Consider a structure $\mathcal N=(N,+,0,<,\ldots)$ elementarily equivalent to an almost o-minimal expansion of an ordered group $\mathcal M=(M,+,0,<,\ldots)$. 
We first reduce to the case in which the $A_\lambda$ are definable without parameters for all $\lambda \in \Lambda$.
There exist parameters $\overline{c} \in N^p$ and first-order formulae $\varphi_{\lambda}(x,y,\overline{c})$ with parameters $\overline{c}$ defining the definable sets $A_{\lambda}$ for all $\lambda \in \Lambda$.
Set $A'_\lambda =\{(z,x,y) \in N^p \times N^m \times N^n\;|\; \mathcal N \models \varphi_\lambda(x,y,z)\}$.
If the corollary holds true for the family $\{A'_\lambda\}_{\lambda \in \Lambda}$, the corollary also holds true for the family $\{A_\lambda\}_{\lambda \in \Lambda}$ because $A_\lambda$ is the fiber $(A'_\lambda)_{\overline{c}}=\{(x,y) \in N^m \times N^n\;|\; (\overline{c},x,y) \in A'_\lambda\}$.
Hence, we may assume that the $A_\lambda$ are definable without parameters for all $\lambda \in \Lambda$.
Let $\varphi_\lambda(x,y)$ denote the first-order formulae without parameters defining the definable sets $A_\lambda$.

Let $A_\lambda^{\mathcal M}$ be the definable subset of $M^{m+n}$ defined by the formula $\varphi_\lambda(x,y)$ for each $\lambda \in \Lambda$.
By Theorem \ref{thm:main}, there exist an open box $B^{\mathcal M}$ in $M^n$ containing the origin and a partition into definable sets 
\begin{equation*}
M^m \times B^{\mathcal M} = X_1^{\mathcal M} \cup \ldots \cup X_k^{\mathcal M}
\end{equation*}
such that the fibers $(X_i^{\mathcal M})_b$ are definable cells of a fixed type for all $b \in M^m$ with $(X_i^{\mathcal M})_b \not= \emptyset$ and either $X_i^{\mathcal M} \subseteq A_\lambda^{\mathcal M}$ or $X_i^{\mathcal M} \cap A_\lambda^{\mathcal M} = \emptyset$ for all $1 \leq i \leq k$.
There exist parameters $\overline{d} \in M^p$ and first-order formulas $\psi_i(x,y,\overline{d})$ with parameters $\overline{d}$ defining the definable sets $X_i^{\mathcal M}$ for all $1 \leq i \leq k$.

Using the first-order formulas $\psi_i(x,y,\overline{d})$, the condition that
\begin{itemize}
\item there exists an open box $B^{\mathcal M}$ in $M^n$ containing the origin and
\item $M^m \times B^{\mathcal M} = X_1^{\mathcal M} \cup \ldots \cup X_k^{\mathcal M}$
\end{itemize}
can be expressed by a first-order sentence $\Phi(\overline{d})$ with parameters $\overline{d}$.
Let $\Psi_i(\overline{d})$ be the sentence expressing the condition $X_i^{\mathcal M} \subseteq A_\lambda^{\mathcal M}$ or $X_i^{\mathcal M} \cap A_\lambda^{\mathcal M} = \emptyset$ for any $1 \leq i \leq k$.
The condition for the fiber $(X_i^{\mathcal M})_b$ being either a cell or an empty set for any $b \in M^m$ is expressed by a first-order formula $\Pi_i(\overline{d})$ with parameters $\overline{d}$ by Lemma \ref{lem:last}.
We have
\begin{equation*}
\mathcal M \models \Phi(\overline{d}) \wedge \bigwedge_{i=1}^k \left( \Psi_i(\overline{d}) \wedge \Pi_i(\overline{d}) \right)
\end{equation*}
by the definitions of $\Phi(\overline{d})$, $\Psi_i(\overline{d})$ and $\Pi_i(\overline{d})$.
We therefore get
\begin{equation*}
\mathcal M \models \exists \overline{d}\ \Phi(\overline{d}) \wedge \bigwedge_{i=1}^k \left( \Psi_i(\overline{d}) \wedge \Pi_i(\overline{d}) \right) \text{.}
\end{equation*}
Since $\mathcal N$ is elementarily equivalent to $\mathcal M$, we finally obtain 
\begin{equation*}
\mathcal N \models \exists \overline{d}\ \Phi(\overline{d}) \wedge \bigwedge_{i=1}^k \left( \Psi_i(\overline{d}) \wedge \Pi_i(\overline{d}) \right) \text{.}
\end{equation*}
Take $\overline{d'} \in N^p$ satisfying the above condition and set $X_i=\{(x,y) \in N^m \times N^n\;|\; \mathcal M \models \psi_i(x,y,\overline{d'})\}$ for all $1 \leq i \leq k$.
Then, there exists an open box $B$ in $N^n$ containing the origin such that the partition $N^m \times B=X_1 \cup \ldots \cup X_k$ is the desired partition.
\end{proof}

\begin{corollary}
A structure elementarily equivalent to an almost o-minimal expansion of an ordered group is a uniformly locally o-minimal structure of the first kind.
\end{corollary}
\begin{proof}
The corollary immediately follows from Theorem \ref{thm:udcd}.
\end{proof}

\end{document}